\documentclass[11pt]{amsart}

\usepackage{graphicx}
\usepackage{amssymb}
\usepackage{amsmath}
\usepackage{amsthm}
\usepackage{xcolor}
\usepackage{enumerate}
\usepackage{float}

\usepackage{calrsfs}
\DeclareMathAlphabet{\pazocal}{OMS}{zplm}{m}{n}

\newcommand{\beq}{\begin{equation}}
\newcommand{\eeq}{\end{equation}}
\newcommand{\mb}[1]{\ensuremath{\mathbb{#1}}}
\newcommand{\N}{\mb{N}}

\newcommand{\R}{\mb{R}}
\newcommand\Rn{{\mathbb R}^n}
\newcommand{\C}{\mb{C}}

\newcommand{\Dx}{\langle D_x \rangle}

\newcommand{\eps}{\varepsilon}

\newcommand{\lara}[1]{\langle #1 \rangle}

\newcommand{\tbf}{\textbf}
\newcommand{\pac}{\pazocal}

\newtheorem{theorem}{Theorem}[section]
\newtheorem{lemma}[theorem]{Lemma}
\newtheorem{proposition}[theorem]{Proposition}
\newtheorem{corollary}[theorem]{Corollary}
\newtheorem{remark}{Remark}[section]
\newtheorem{definition}{Definition}[section]
\newtheorem{example}{Example}[section]

\DeclareMathOperator{\adj}{\textbf{adj}}
\DeclareMathOperator{\cof}{\textbf{cof}}

\DeclareMathOperator{\ima}{Im}

\DeclareMathOperator{\diag}{diag}
\DeclareMathOperator{\tr}{tr}

\title[Hyperbolic systems with multiplicities]{Well--posedness of hyperbolic systems with multiplicities and smooth coefficients}

\author[Claudia Garetto]{Claudia Garetto}
\address{
	Department of Mathematical Sciences
	\endgraf
	Loughborough University
	\endgraf
	Loughborough, Leicestershire, LE11 3TU
	\endgraf
	United Kingdom
	\endgraf
	{\it E-mail address} {\rm c.garetto@lboro.ac.uk}
}

\author[Christian J\"ah]{Christian J\"ah}
\address{
	Department of Mathematical Sciences
	\endgraf
	Loughborough University
	\endgraf
	Loughborough, Leicestershire, LE11 3TU
	\endgraf
	United Kingdom
	\endgraf
	{\it E-mail address} {\rm c.jaeh@lboro.ac.uk}
}

\thanks{The first author was partially and the second author fully supported by the EPSRC grant EP/L026422/1}

\date{}


\usepackage{hyperref}

\begin{document}

\maketitle

\begin{abstract}
	We study hyperbolic systems with multiplicities and smo\-oth coefficients. In the case of non-analytic, smooth coefficients, we prove well-posedness in any Gevrey class and when the coefficients are analytic, we prove $C^\infty$ well-posedness.  The proof is based on a transformation to block Sylvester form introduced by D'Ancona and Spagnolo in \cite{DS} which increases the system size but does not change the eigenvalues. This reduction introduces lower order terms for which appropriate Levi-type conditions are found. These translate then into conditions on the original coefficient matrix. This paper can be considered as a generalisation of \cite{GR}, where weakly hyperbolic higher order equations with lower order terms were considered.
\end{abstract}

\setcounter{tocdepth}{1}
\tableofcontents
 
\section{Introduction}

We consider the Cauchy problem \begin{equation} \label{CauchyP} 
\left\{ \begin{array}{ll}
& D_t u - A(t,D_x)u = 0,~(t,x) \in [0,T] \times \R^n,\\
& \left. u \right|_{t=0} = u_0,~x \in \R^n,
\end{array} \right.
\end{equation} 
where $D_t=-{\rm i}\partial_t$, $D_x=-{\rm i}\partial_x$, and $A(t,D_x)$ is an $m \times m$ matrix of first-order differential operators with time-dependent coefficients and $u$ is a column vector with components $u_1$, $\dots$, $u_m$. We assume that \eqref{CauchyP} is hyperbolic, whereby we mean that the matrix $A(t,\xi)$ has only real eigenvalues. These eigenvalues, rescaled to order $0$ by multiplying by $\lara{\xi}^{-1}$, will be denoted by  $\lambda_1(t,\xi),\dots,\lambda_m(t,\xi)$. Following Kinoshita and Spagnolo in \cite{KS}, we assume throughout this paper that there exists a positive constant $C$ such that \begin{equation} \label{eq:CondKinSp} 
\lambda_i^2(t,\xi) + \lambda_j^2(t,\xi) \leq C (\lambda_i(t,\xi) - \lambda_j(t,\xi))^2,~(t,\xi) \in [0,T] \times \R^n
\end{equation} for all $1 \leq i < j \leq m$.

As observed in \cite{GR:15} combining the well-posedness results in \cite{KY:06,Yu:05} we already know that the Cauchy problem \eqref{CauchyP} is well-posed in the Gevrey class $\gamma^s$, with 
\[
1\le s<1+\frac{1}{m}
\]
as well as in the corresponding spaces of (Gevrey-Beurling) ultradistributions. In this paper we want to prove that when $A(t,D_x)$ has smooth coefficients and the condition \eqref{eq:CondKinSp} on the eigenvalues holds, then the Gevrey well-posedness result above can be extended to any $s \geq 1$. Since, by the results of Kajitani and Yuzawa when $s\ge 1+\frac{1}{m}$ at least an ultradistributional solution to the Cauchy problem \eqref{CauchyP} exists, we will prove that this solution does actually belong to the Gevrey class $\gamma^s$. In the case of analytic coefficients, we will prove instead that the Cauchy problem \eqref{CauchyP} is $C^\infty$ well-posed.

In this paper we assume that the Gevrey classes $\gamma^s(\Rn)$ are well-known: these are spaces of all
 $f\in C^\infty(\R^n)$ such that for every compact set $K\subset\R^n$ there
exists a constant $C>0$ such that for all $\beta\in\mathbb N_0^n$ we have the estimate
\[
\sup_{x\in K}|\partial^\beta f(x)|\le C^{|\beta|+1}(\beta!)^s.
\]
For $s=1$, we obtain the class of analytic functions.
We refer to 
\cite{GarRuz:1} for a detailed discussion and Fourier characterisations of Gevrey spaces of different types and the definition of the corresponding spaces of ultradistributions. 

The well-posedness of hyperbolic equations and systems with multiplicities has been a challenging problem for a long time. In the last decades several results have been obtained for scalar equations with $t$-dependent coefficients (\cite{ColKi:02,ColKi:02-2,COr,CDS,dAKi:05,GarRuz:1,GR,GarRuz:3,KS}, to quote a few)  but the research on hyperbolic systems with multiplicities has not been as successful. We mention here the work of D'Ancona, Kinoshita and Spagnolo \cite{dAKi:04} on weakly hyperbolic systems (i.e. systems with multiplicities) of size $2\times 2$ and $3\times 3$ with H\"older dependent coefficients later generalised to any matrix size by Yuzawa in \cite{Yu:05} and to $(t,x)$-dependent coefficients by  Kajitani and Yuzawa in \cite{KY:06}. In all these papers, well-posedness is obtained in Gevrey classes of a certain order depending on the regularity of the coefficients and the system size. Systems of this type have recently also been investigated in \cite{G:15,GR:15}.

It is a natural question to ask if under stronger assumptions on the regularity of the coefficients, for instance smooth or analytic coefficients, the well-posedness of the corresponding Cauchy problem could be improved, in the sense if one could get well-posedness in every Gevrey class or $C^\infty$--well-posedness. It is known that this is possible for scalar equations under suitable assumptions on the multiple roots and Levi conditions on the lower order terms, see \cite{GR,KS} for $C^k$ and $C^\infty$ coefficients and \cite{GR,JT,KS} for analytic coefficients. This paper gives a positive answer to this question by extending the results for scalar equations in \cite{GR,KS} to systems with multiplicities. This will require a transformation of the system in \eqref{CauchyP} into block-diagonal form with Sylvester blocks which increases the system size from $m\times m$ to $m^2\times m^2$ but does not change the eigenvalues, in the sense that every block will have the same eigenvalues as $A(t,\xi)$. Such a transformation, introduced by D'Ancona and Spagnolo in \cite{DS}, has the side effect to generate a matrix of lower order terms even when the original system is homogeneous, i.e., \eqref{CauchyP} will be transformed into a Cauchy problem of the type
\[
\left\{ \begin{array}{l}
	D_t U = \pazocal A(t,D_x) U + \pazocal B(t,D_x)U, \\
	\left. U \right|_{t=0} = U_0.
	\end{array} \right.
\]
It becomes therefore crucial to understand how the lower order terms in $\pazocal{B}(t,\xi)$ are related to the matrix $\pazocal{A}(t,\xi)$, which is in turn related to $A(t,\xi)$, and which Levi-type conditions have to be formulated on them to get the desired well-posedness. These Levi-type conditions will then be expressed in terms of the matrix $A(t,\xi)$. In the next subsection we collect our main results and we give a scheme of the proof.

\subsection{Results and scheme of the proof} \label{sec:ResultScheme}

In the sequel, we denote the elementary symmetric polynomials $\sigma_h^{(m)}(\lambda)$ by
\[
\sigma_h^{(m)}(\lambda)=(-1)^h\sum_{1\le i_1<...<i_h\le m}\lambda_{i_1}...\lambda_{i_h},
\]
for $1 \leq h \leq m$ and $\sigma_0^{(m)}(\lambda) = 1$, where $\lambda=(\lambda_1,\dots,\lambda_m)$ is given by the rescaled eigenvalues $\lambda_i = \lambda_i(t,\xi)$ of $A(t,\xi)$ and $\pi_i\lambda=(\lambda_1,...,\lambda_{i-1},\lambda_{i+1},...,\lambda_m)$. Moreover, given $f=f(t,\xi)$ and $g(t,\xi)$ we use the notation $f\prec g$, when it exists a constant $C>0$ such that $f(t,\xi)\le C g(t,\xi)$ for all $t\in[0,T]$ and $\xi\in\R^n$. We will also use $(\cdot)$ in the upper left corner of a symbol as in $b_{ij}^{(l)}$. By that we will not denote derivatives but use this as an index.

\begin{theorem} \label{theo_GJ_intro}
Let $A(t,D_x)$, $t\in[0,T]$, $x\in\R^n$, be an $m\times m$ matrix of first order differential operators with $C^\infty$-coefficients. Let $A(t,\xi)$ have real eigenvalues satisfying condition \eqref{eq:CondKinSp}. Assume that the Cauchy problem 
\[
\left\{ \begin{array}{l}
	D_t U = \pazocal A(t,D_x) U + \pazocal B(t,D_x)U, \\
	\left. U \right|_{t=0} = U_0,
	\end{array} \right.
\]
obtained from \eqref{CauchyP} by block Sylvester transformation has the lower order terms matrix $ \pazocal B(t,\xi)$ with entries $b_{kj}^{(l)}(t,\xi)$ fulfilling the Levi-type conditions \begin{equation} \label{cond:LeviPrec}
\sum_{k=1}^m|b_{kj}^{(l)}(t,\xi)|^2\prec \sum_{i=1}^m |\sigma_{m-l}^{(m-1)}(\pi_i \lambda)|^2,
\end{equation} for $l=1,\dots,m-1$ and $j=1,\dots,m$. Hence, for all $s\ge 1$ and for all $u_0\in \gamma^s(\R^n)^m$ there exists a unique solution $u\in C^1([0,T], \gamma^s(\R^n))^m$ of the Cauchy problem \eqref{CauchyP}.
\end{theorem}

The formulation of the Levi-type conditions given above requires a precise knowledge of the matrix $\pazocal{B}(t,\xi)$. For that see the section \ref{sec:StructBtxi}. It is possible to state the previous well-posedness result completely in terms of the matrix $A(t,\xi)$ and the Cauchy problem \eqref{CauchyP}. This means to introduce an additional hypothesis on the coefficients of $A(t,\xi)$ which implies the Levi-type conditions on $\pazocal{B}(t,\xi)$. In the final section of the paper we will prove that in some cases, for instance when $m=2$, this second formulation is equivalent to the one given in Theorem \ref{theo_GJ_intro}.

\begin{theorem} \label{theo_GJ_2_intro}
Let $A(t,D_x)$, $t\in[0,T]$, $x\in\R^n$, be an $m\times m$ matrix of first order differential operators with $C^\infty$-coefficients. Let $A(t,\xi)$ have real eigenvalues satisfying condition \eqref{eq:CondKinSp} and let $Q=(q_{ij})$ be the symmetriser of $A_0=\lara{\xi}^{-1}A$. Assume that \begin{equation} \label{eq:LeviIntro}
\max_{k=1,\dots,m-1}\Vert D_t^k A_0(t,\xi)\Vert^2 \prec q_{j,j}(t,\xi)
\end{equation} for all $j=1,\dots,m-1$ and all $(t,\xi) \in [0,T] \times \R^n$. Hence, for all $s\ge 1$ and for all $u_0\in \gamma^s(\R^n)^m$ there exists a unique solution $u\in C^1([0,T], \gamma^s(\R^n))^m$ of the Cauchy problem \eqref{CauchyP}. Here, $\| \cdot\|$ denotes the standard matrix norm. \end{theorem}

\begin{remark} For some more concrete examples in the cases $m = 2$ and $3$, see the remarks in Section \ref{sec_final}. \end{remark}

Since the entries of the symmetriser are polynomials depending on the eigenvalues of $A(t,\xi)$, we require in Theorem \ref{theo_GJ_2_intro} that the $t$-derivatives of $A(t,\xi)$ up to order $m-1$ are bounded by suitable polynomials of the eigenvalues $\lambda_1(t,\xi),\dots,\lambda_m(t,\xi)$. Note that, as observed already in the appendix of \cite{GR}, these polynomials can be expressed in terms of the entries of $A(t,\xi)$.

When the entries of $A(t,\xi)$ are analytic, then we prove that the Cauchy problem \eqref{CauchyP} is $C^\infty$ well-posed. The precise statements can be obtained by replacing $\gamma^s$ with $C^\infty$ in Theorem \ref{theo_GJ_intro} and Theorem \ref{theo_GJ_2_intro}.

We conclude this subsection by presenting the scheme of the proof of Theorem \ref{theo_GJ_intro} which combines ideas from \cite{DS} and \cite{GR} .

 \begin{enumerate}[{\quad Step~}1{~}]
	\item Compute the adjunct matrix $\adj(I_m\tau - A(t,\xi)) = \cof(I_m\tau - A^T(t,\xi))$, where $I_m$ is the identity matrix of size $m \times m$. We thus have the relation \begin{equation*}
	\adj(I_m \tau - A(t,\xi))(I_m \tau - A(t,\xi)) = \sum\limits_{h=0}^{m} c_{h}(t,\xi)I_m \tau^{m-h},
	\end{equation*} where the $c_{h}(t,\xi)$ are homogeneous polynomials of order $h$ in $\xi$ and are given by the coefficients of the characteristic polynomial of $A(t,\xi)$. See Appendix \ref{sec:Appendix}.
	\item Apply the operator $\adj(I_m D_t-A(t,D_x))$, associated to the symbol $\adj(I_m \tau - A(t,\xi))$, to the system \eqref{CauchyP} and obtain a set of scalar equations for $u_1$ to $u_m$, where the operator acting on these is associated to $\det(I_m \tau - A(t,\xi))$. Additionally, one gets some lower order terms which can be computed explicitly.
	\item Convert the resulting set of equations \begin{equation*}
	\det(I_m D_t - A(t,D_x))u + l.o.t. = 0
	\end{equation*} to Sylvester block diagonal form following the method of Taylor in \cite{Taylor:81}, i.e by setting \begin{equation} \label{eq:DefTrafo} \begin{aligned}
			U &= \left( U_1 , U_2, \dots, U_m \right)^T,~\text{where} \\
			U_{k} &= \left( \Dx^{m-1} u_k, D_t\Dx^{m-2} u_k, \dots, D_t^{m-1} u_k \right)
	\end{aligned} \end{equation} for $k=1,\dots,m$. This transformation maps each equation to a system in Sylvester form and glues those systems in block diagonal form together. Hence, we achieve a block diagonal form with Sylvester blocks associated to the characteristic polynomial of \eqref{CauchyP}. This means that each block will have the same eigenvalues as $A(t,\xi)$. The initial data will be transformed in the same way to obtain a new set of initial data $U_0$ for the new system.
	\item Consider the resulting system \begin{equation} \label{eq:TrafoIntro}
		\left\{ \begin{array}{ll}
		&D_t U = \pac A(t,D_x)U + \pac B(t,D_x)U \\
		&\left. U \right|_{t=0} = U_0
		\end{array} \right.,
	\end{equation} where $\pazocal A(t,D_x)$ and $\pazocal B(t,D_x)$ are matrices of size $m^2 \times m^2$ with a special structure. As explained above, $\pazocal A(t,D_x)$ is a block diagonal matrix with $m$ identical blocks of Silvester matrices having the same eigenvalues as $A(t,\xi)$ and $\pazocal B(t,D_x)$ is composed of $m \times m^2$ blocks with only the last row not identically zero. Since the original homogeneous system has been transformed into a system with lower order terms, to get well-posedness of the corresponding Cauchy problem \eqref{eq:TrafoIntro}, we need to find some Levi-type conditions. These are obtained by following the ideas for scalar equations in \cite{GR}.
	\item We apply the partial Fourier transform with respect to $x$ to \eqref{eq:TrafoIntro} and we prove an energy estimate from which the assertions of the well-posedness theorems follow in a standard way. A key point is the construction of the quasi-symmetriser of the matrix $\pac{A}(t,\xi)$.
\end{enumerate}

The remainder of the paper is organised as follows. In Section \ref{SEC:qs}, we present a short survey on the quasi-symmetriser which will be employed to formulate and prove the energy estimate. The core of Section \ref{sec:BlockDiagonal} is the transformation of $A(t,\xi)$ from \eqref{CauchyP} to block Sylvester form. An explicit description of $\adj(I_m D_t - A(t,D_x))$ and the resulting lower order terms is also given in Section \ref{sec:BlockDiagonal}, together with a detailed scheme of the proof in the cases $m=2$ and $m=3$. Section \ref{sec_energy} is devoted to the energy estimate and Section \ref{sec_lot} to the estimates for the lower order terms. The paper ends with the well-posedness results in Section \ref{sec_final} and the Appendix L, where we collect some algebraic results concerning $\adj(I_m \tau - A(t,\xi))$.

\section{The quasi-symmetriser} \label{SEC:qs}

Here we recall some facts about the quasi-symmetriser that we will need throughout the paper. For more details see \cite{DS,KS}. Note that for $m\times m$ matrices $A_1$ and $A_2$ the notation $A_1\le A_2$ means $(A_1v,v)\le (A_2v,v)$ for all $v\in\C^m$ with $(\cdot,\cdot)$ the scalar product in $\C^m$. Let $M(\lambda)$ be a $m\times m$ Sylvester matrix with real eigenvalues $\lambda_l$, i.e.,
\[
M(\lambda)=\left(
\begin{array}{ccccc}
0 & 1 & 0 & \dots & 0\\
0 & 0 & 1 & \dots & 0 \\
\dots & \dots & \dots & \dots & 1 \\
-\sigma_m^{(m)}(\lambda) & -\sigma_{m-1}^{(m)}(\lambda) & \dots & \dots & -\sigma_1^{(m)}(\lambda) \\
\end{array}
\right),
\]
where the $\sigma_h^{(m)}(\lambda)$ are defined as \begin{equation} \label{eq:DefSymPol}
	\sigma_h^{(m)}(\lambda)=(-1)^h\sum_{1\le i_1<...<i_h\le m}\lambda_{i_1}...\lambda_{i_h}
\end{equation} for all $1\le h\le m$. We further set $\sigma_0^{(m)}(\lambda) = 1$. In the sequel we make use of the following notations: $\pazocal{P}_m$ for the class of permutations of $\{1,...,m\}$, $\lambda_\rho=(\lambda_{\rho_1},...,\lambda_{\rho_m})$ with $\lambda\in\R^m$ and $\rho\in\pazocal{P}_m$, $\pi_i\lambda=(\lambda_1,...,\lambda_{i-1},\lambda_{i+1},...,\lambda_m)$ and $\lambda'=\pi_m\lambda=(\lambda_1,...,\lambda_{m-1})$.

To construct the quasi-symmetriser, we follow \cite{KS} and define $P^{(m)}(\lambda)$ inductively by $P^{(1)}(\lambda)=1$ and \begin{equation*}
P^{(m)}(\lambda)=\left(
\begin{array}{ccccc}
\, & \, & \, & \, & 0\\
\, & \, & P^{(m-1)}(\lambda') & \, & \vdots \\
\, & \, & \, & \, & 0 \\
\sigma_{m-1}^{(m-1)}(\lambda') & \dots & \dots & \sigma_1^{(m-1)}(\lambda') & 1 \\
\end{array}
\right).
\end{equation*} Further, we set, for $\eps\in(0,1]$, \begin{equation*}
	P_\eps^{(m)}(\lambda)=H^{(m)}_\eps P^{(m)}(\lambda),
\end{equation*} where $H_\eps^{(m)}={\rm diag}\{\eps^{m-1},...,\eps,1\}$. We remark that $P^{(m)}(\lambda)$ depends only on $\lambda'$. Finally, the quasi-symmetriser is the Hermitian matrix \begin{equation*}
	Q^{(m)}_\eps(\lambda)=\sum_{\rho\in\pazocal{P}_m} P_\eps^{(m)}(\lambda_\rho)^\ast P_\eps^{(m)}(\lambda_\rho).
\end{equation*} To describe the properties of $Q^{(m)}_\eps(\lambda)$ in more detail in the next proposition, we denote by $W^{(m)}_i(\lambda)$ the row vector \begin{equation*}
\big(\sigma_{m-1}^{(m-1)}(\pi_i\lambda),...,\sigma_1^{(m-1)}(\pi_i\lambda),1\big),\quad 1\le i\le m,
\end{equation*} and let $W^{(m)}(\lambda)$ be the matrix with row vectors $W^{(m)}_i$.

The following proposition collects the main properties of the quasi-symmetriser $Q^{(m)}_\eps(\lambda)$. For a detailed proof we refer the reader to Propositions 1 and 2 in \cite{KS} and to Proposition 1 in \cite{DS}. \pagebreak

\begin{proposition}
\label{prop_qs}
\leavevmode
\begin{enumerate}[{\qquad}(i){}]
\item The quasi-symmetriser $Q_\eps^{(m)}(\lambda)$ can be written as
\[
Q_0^{(m)}(\lambda)+\eps^2 Q_1^{(m)}(\lambda)+...+\eps^{2(m-1)}Q_{m-1}^{(m)}(\lambda),
\]
where the matrices $Q^{(m)}_i(\lambda)$, $i=1,...,m-1,$ are non-negative and Hermitian with
entries being symmetric polynomials in $\lambda_1,...,\lambda_m$.
\item There exists a function $C_m(\lambda)$ bounded for
bounded $|\lambda|$ such that
		\[
		C_m(\lambda)^{-1}\eps^{2(m-1)}I\le Q^{(m)}_\eps(\lambda)\le C_m(\lambda)I.
		\]
		\item We have
		\[
		-C_m(\lambda)\eps Q_\eps^{(m)}(\lambda)\le Q_\eps^{(m)}(\lambda) M(\lambda)- M(\lambda)^\ast Q_\eps^{(m)}(\lambda)\le C_m(\lambda)\eps Q_\eps^{(m)}(\lambda).
		\]
		\item For any $(m-1)\times(m-1)$ matrix $T$ let $T^\sharp$ denote the $m\times m$ matrix
		\[
		\left(
		\begin{array}{cc}
		T & 0\\
		0 & 0 \\
		\end{array}
		\right).
		\]
		Then, $Q_\eps^{(m)}(\lambda)=Q_0^{(m)}(\lambda)+\eps^2\sum_{i=1}^m Q_\eps^{(m-1)}(\pi_i\lambda)^\sharp$.
		\item We have
		\[
		Q_0^{(m)}(\lambda)=(m-1)! W^{(m)}(\lambda)^\ast W^{(m)}(\lambda).
		\]
		\item We have
		\[
		\det Q_0^{(m)}(\lambda)=(m-1)!\prod_{1\le i<j\le m}(\lambda_i-\lambda_j)^2.
		\]
		\item There exists a constant $C_m$ such that
		\[
		q_{0,11}^{(m)}(\lambda)\cdots q_{0,mm}^{(m)}(\lambda)\le C_m\prod_{1\le i<j\le m}(\lambda^2_i+\lambda^2_j).
		\]
		\end{enumerate}
		\end{proposition}
		
We finally recall that a family $\{Q_\alpha\}$ of non-negative Hermitian matrices is called \emph{nearly diagonal} if there exists a positive constant $c_0$ such that \begin{equation*}
		Q_\alpha\ge c_0\,{\rm diag}\,Q_\alpha
\end{equation*} for all $\alpha$, with ${\rm diag}\,Q_\alpha ={\rm diag}\{q_{\alpha,11},...,q_{\alpha, mm}\}$. The following linear algebra result is proven in \cite[Lemma 1]{KS}.
		
\begin{lemma} \label{lem_old}
Let $\{Q_\alpha\}$ be a family of non-negative Hermitian $m\times m$ matrices such that $\det Q_\alpha>0$ and	\[
	\det Q_\alpha \ge c\, q_{\alpha,11}q_{\alpha,22}\cdots q_{\alpha,mm}
	\] for a certain constant $c>0$ independent of $\alpha$. Then, \[
	Q_\alpha\ge c\, m^{1-m}\,{\rm diag}\,Q_\alpha
	\] for all $\alpha$, i.e., the family $\{Q_\alpha\}$ is nearly diagonal.
\end{lemma}

Lemma \ref{lem_old} is employed to prove that the family  $Q_\eps^{(m)}(\lambda)$ of quasi-sym\-me\-tri\-sers defined above is nearly diagonal when $\lambda$ belongs to a suitable set. The following statement is proven in \cite[Proposition 3]{KS}.

\begin{proposition}
	\label{prop_SM}
	For any $M>0$ define the set
	\[
	\pazocal{S}_M=\{\lambda\in\R^m:\, \lambda_i^2+\lambda_j^2\le M (\lambda_i-\lambda_j)^2,\quad 1\le i<j\le m\}.
	\]
	Then the family of matrices $\{Q_\eps^{(m)}(\lambda):\, 0<\eps\le 1, \lambda\in\pazocal{S}_M\}$ is nearly diagonal.
\end{proposition}

We conclude this section with a result on nearly diagonal matrices depending on three parameters, $\eps$, $t$, and $\xi$ which will be crucial in Section \ref{sec_energy}. Note that this is a straightforward extension of Lemma 2 in \cite{KS} valid for matrices depending on two parameters, $\eps$ and $t$.
				
\begin{lemma}
	\label{lem_new}
	Let $\{ Q^{(m)}_\eps(t,\xi): 0<\eps\le 1, 0\le t\le T, \xi\in\R^n\}$ be a nearly diagonal family of coercive Hermitian matrices of class ${C}^k$ in $t$, $k\ge 1$. Then, there exists a constant $C_T>0$ such that for any	continuous function $V:[0,T]\times\R^n\to \C^m$ we have
	\[
		\int_{0}^T \frac{|(\partial_t Q^{(m)}_\eps(t,\xi)
		V(t,\xi),V(t,\xi))|}{(Q^{(m)}_\eps(t,\xi)V(t,\xi),V(t,\xi))^{1-1/k}
		|V(t,\xi)|^{2/k}}\, dt\le C_T
		\Vert Q^{(m)}_\eps(\cdot,\xi)\Vert^{1/k}_{{C}^k([0,T])}
	\]
	for all $\xi\in\Rn.$
\end{lemma}

\begin{remark} \label{rem:QuasiSym}
	All results of this section hold true in the when $Q_{\eps}^{(m)}(t,\xi)$ is replaced by a block diagonal matrix $\pac{Q}_{\eps}^{(m)}(t,\xi)$ with $m$ identical matrices $Q_{\eps}^{(m)}(t,\xi)$ on its diagonal. The corresponding block diagonal matrix with $W^{m}(\lambda)$ blocks is denoted by $\pac{W}^{(m)}(\lambda)$. Proofs follow from a block-wise treatment and application of the results above.
\end{remark}

\subsection{The quasi-symmetriser in the case $m=2$ and $m=3$}
\label{Ex_qs}

For the advantage of the reader, we conclude this section by computing the quasi-symmetrisers $Q^{(2)}_{\eps}$ and $Q^{(3)}_\eps$. For $m=2$, we obtain
\begin{eqnarray*}
	W^{(2)}(\lambda) &=& \begin{pmatrix}
		-\lambda_2 & 1 \\
		-\lambda_1 & 1
	\end{pmatrix} \\
	Q_\eps^{(2)}(\lambda) &=& \begin{pmatrix}
		\lambda_1^2 + \lambda_2^2 & -(\lambda_1 + \lambda_2) \\
		-(\lambda_1 + \lambda_2) & 2
	\end{pmatrix} + 2\eps^2 \begin{pmatrix}
	1 & 0 \\
	0 & 0
\end{pmatrix}.
\end{eqnarray*} Similarly, for $m=3$, we obtain \begin{eqnarray*}
 W^{(3)}(\lambda) &=& \begin{pmatrix}
 	\lambda_2 \lambda_3 & -(\lambda_2+\lambda_3) & 1 \\
 	\lambda_3 \lambda_1 & -(\lambda_3+\lambda_1) & 1 \\
 	\lambda_1 \lambda_2 & -(\lambda_1+\lambda_2) & 1
 \end{pmatrix} \\
 Q_\eps^{(3)}(\lambda) &=& 2 \sum\limits_{1 \leq i < j \leq 3} \begin{pmatrix}
 	(\lambda_i\lambda_j)^2 & -\lambda_i\lambda_j(\lambda_i+\lambda_j) & \lambda_i\lambda_j \\
 	-\lambda_i\lambda_j(\lambda_i+\lambda_j) & (\lambda_i + \lambda_j)^2 & -(\lambda_i + \lambda_j) \\
 	\lambda_i\lambda_j & -(\lambda_i+\lambda_j) & 1
 \end{pmatrix} \\
 && \quad + 2\eps^2 \sum\limits_{1 \leq i \leq 3} \begin{pmatrix}
 	\lambda_i^2 & -\lambda_i & 0 \\
 	-\lambda_i & 1 & 0 \\
 	0 & 0 & 0
 \end{pmatrix} + 6\eps^4 \begin{pmatrix}
 1 & 0 & 0 \\
 0 & 0 & 0 \\
 0 & 0 & 0
\end{pmatrix}.
\end{eqnarray*}

\section{Sylvester block diagonal reduction} \label{sec:BlockDiagonal}

This section is devoted to the Sylvester block diagonal reduction that will be employed on the system \eqref{CauchyP}. This transformation has been introduced by D'Ancona and Spagnolo in \cite{DS}. Here we give a detailed description of how this reduction works on the system $I_m D_t - A(t,D_x)$ and we present explicit formulas for the matrix of lower order terms generated by the procedure. Note that these results are obtained from general linear algebra statements that are collected in the appendix at the end of the paper. We will refer to Appendix \ref{sec:Appendix} throughout this section. The subsections refer to the steps of the proof outlined in Subsection \ref{sec:ResultScheme}.

\subsection{Step 1: The adjunct $\adj(I_mD_t - A(t,D_x))$}

A straightforward application of Lemma \ref{lem:AdjunctRep} leads us to the following proposition.

 \begin{proposition} \label{prop:FormulaLOT} 
 Let $I_m D_t - A(t,D_x)$ be the operator in \eqref{CauchyP}. Then,  
 \begin{equation*}
	\adj(I_m D_t - A(t,D_x)) = \sum\limits_{h = 0}^{m-1} \tbf{A}_{h}(t,D_x) D_t^{m-1-h}
	\end{equation*} where \begin{equation} \label{eq:AuxEqDef1}
	\tbf{A}_{h}(t,D_x) = \sum\limits_{h'=0}^{h} \sigma_{h'}^{(m)}(\lambda) A^{h-h'}(t,D_x),
	\end{equation} 
$\lambda = (\lambda_1, \dots, \lambda_m)$ and $\sigma_h^{(m)}(\lambda)$ as defined in \eqref{eq:DefSymPol}. The differential operator $\adj(I_m D_t - A(t,D_x))$ is of order $m-1$ with respect to $D_t$ and every differential operator $\tbf{A}_h(t,D_x)$, $1 \leq h \leq m$, is of order $h$ with respect to $D_x$. We set $A^0(t,D_x) = I_m$.
\end{proposition}

Proposition \ref{prop:FormulaLOT} completes Step $1$ of our proof as outlined in the scheme. We can therefore proceed to Step 2.   

\subsection{Step 2: Computation of the lower order terms}
 
 \begin{proposition} 
 \label{prop_B_lot}
 
 The lower order terms that arise after applying the adjunct $\adj(I_mD_t-A(t,D_x))$ to the original operator $I_mD_t -A(t,D_x)$ are given by \begin{equation} \label{eq:LemLOT}
	B(t,D_t,D_x)u = - \sum\limits_{h =0}^{m-2} \tbf{A}_{h}(t,D_x)\tbf{A}'_{h}(t,D_t,D_x) ,
	\end{equation} where $\tbf{A}_{h}(t,D_x)$ is defined in \eqref{eq:AuxEqDef1} and \begin{equation} \label{eq:auxDef2}
	\tbf{A}'_{h}(t,D_t,D_x) = \sum_{h'=h}^{m-2} \binom{m-1-h}{h'+1-h} (D_t^{h'+1-h}A)(t,D_x)D_t^{m-2-h'}u.
	\end{equation}
\end{proposition}

\begin{proof} From Proposition \ref{prop:FormulaLOT} and Leibniz rule, we have 
\begin{equation} \label{eq:ComputeLOT1}
	\begin{aligned}
	& \adj(I_m D_t - A(t,D_x))(I_m D_t u - A(t,D_x)u) \\
	& \qquad =  \sum\limits_{h = 0}^{m-1} \tbf{A}_{h}(t,D_x) D_t^{m-1-h} (I_m D_t u - A(t,D_x)u) \\
	& \qquad =  \sum\limits_{h = 0}^{m-1} \tbf{A}_{h}(t,D_x) D_t^{m-h} u -  \sum\limits_{h =0}^{m-1} \tbf{A}_{h}(t,D_x) D_t^{m-1-h} (A(t,D_x)u) \\
	& \qquad = \sum\limits_{h =0}^{m-1} \tbf{A}_{h}(t,D_x) D_t^{m-h} u \\
	& \qquad \quad - \sum\limits_{h =0}^{m-1} \tbf{A}_{h}(t,D_x) \sum_{h'=0}^{m-1-h} \binom{m-1-h}{h'}(D_t^{h'}A)(t,D_x)D_t^{m-1-h-h'}u.
	\end{aligned}
	\end{equation} 
	Now we write the second summand in the last equation in \eqref{eq:ComputeLOT1} as $Xu+Yu$ where $Xu$ contains all terms with $h'=0$ and \begin{equation}
	\begin{aligned}
	\label{formula_Y}
		Yu & = -\sum\limits_{h =0}^{m-1} \tbf{A}_{h}(t,D_x) \sum_{h'=1}^{m-1-h} \binom{m-1-h}{h'}(D_t^{h'}A)(t,D_x)D_t^{m-1-h-h'}u\\
		&= -\sum\limits_{h =0}^{m-2} \tbf{A}_{h}(t,D_x) \sum_{h'=1}^{m-1-h} \binom{m-1-h}{h'}(D_t^{h'}A)(t,D_x)D_t^{m-1-h-h'}u.
		\end{aligned}
	\end{equation} 
	By replacing $h'$ with $h'+1-h$ in the second sum in \eqref{formula_Y} we get
	\[
	Yu=-\sum\limits_{h =0}^{m-2} \tbf{A}_{h}(t,D_x) \sum_{h'=h}^{m-2} \binom{m-1-h}{h'+1-h} (D_t^{h'+1-h}A)(t,D_x)D_t^{m-2-h'}u
	\]
and then by \eqref{eq:auxDef2} we conclude that $Yu=B(t,D_t,D_x)u$ as desired. It remains to show that \begin{equation} \label{eq:aux6}
		\sum\limits_{h =0}^{m-1} \tbf{A}_{h}(t,D_x) D_t^{m-h} u + Xu = \det(I_mD_t - A(t,D_x))u.
	\end{equation} 
	By \eqref{eq:AuxEqDef1}, we obtain \begin{equation*}
	\tbf{A}_{h}(t,D_x) A(t,D_x) = \tbf{A}_{h+1}(t,D_x) - \sigma_{h+1}^{(m)}(\lambda) I_m
	\end{equation*} and, thus, \begin{equation*}\begin{aligned}
	X &= -\sum_{h=0}^{m-1} \tbf{A}_h(t,D_x)A(t,D_x)D_t^{m-1-h} \\
	&= -\sum_{h=1}^{m} \tbf{A}_{h}(t,D_x)D_t^{m-h} + \underbrace{\sum_{h=1}^{m}\sigma_{h}^{(m)}(\lambda) I_m D_t^{m-h}}_{=\det(I_m D_t - A(t,D_x))-I_mD_t^m\,(\text{see}\, \eqref{eq:CharPolDef})}.
	\end{aligned}
	\end{equation*} Using that $\tbf{A}_m =0$ (thanks to the Cayley-Hamilton theorem, see \eqref{eq:CH}) and $\tbf{A}_0 = I_m$, we obtain \eqref{eq:aux6} which concludes the proof. $\blacksquare$ \end{proof}

It will be convenient for the description of some important matrices in this paper to rewrite the lower order terms in a different way. More precisely, we have the following corollary. 
\begin{corollary} \label{cor:LOT}
	We can write the lower order term in \eqref{eq:LemLOT} as
	\begin{equation} \label{eq:GoodLOTdef}
	B(t,D_t,D_x) = -\sum_{h=0}^{m-2} \tbf{B}_{h+1}(t,D_x)D_t^{h},
	\end{equation} where \begin{equation} \label{eq:GoodLOTdef2}
	\tbf{B}_{h+1}(t,D_x) = \sum_{h'=0}^{m-2-h} \binom{m-1-h'}{h} \tbf{A}_{h'}(t,D_x)(D_t^{m-1-h-h'}A)(t,D_x)
	\end{equation} and $\tbf{A}_{h'}(t,D_x)$ is given by \eqref{eq:AuxEqDef1}.
\end{corollary}
\begin{proof}
	Formula \eqref{eq:GoodLOTdef} follows from \eqref{eq:LemLOT} by interchanging the order of the sums appropriately. Indeed, we have, using \eqref{eq:AuxEqDef1} and \eqref{eq:auxDef2}, that \begin{equation} \label{eq:ComputeBDt}
	\begin{aligned}
	& B(t,D_t,D_x) \\
	&\quad = - \sum_{h=0}^{m-2} \tbf{A}_h(t,D_x) \sum_{h'=h}^{m-2} \binom{m-1-h}{h'+1-h} (D_t^{h'+1-h}A)(t,D_x)D_t^{m-2-h'} \\
	&\quad = -\sum_{h'=0}^{m-2} \underbrace{\sum_{h=0}^{h'} \tbf{A}_h(t,D_x) \binom{m-1-h}{h'+1-h} (D_t^{h'+1-h}A)(t,D_x)}_{=: \tbf{B}_{m-1-h'}(t,D_x)} D_t^{m-2-h'} \\
	& \quad = - \sum_{h=0}^{m-2} \tbf{B}_{h+1}(t,D_x)D_t^h ,
	\end{aligned}
	\end{equation} with \begin{equation*}
		\tbf{B}_{h+1}(t,D_x) = \sum_{h'=0}^{m-2-h} \binom{m-1-h'}{h} \tbf{A}_{h'}(t,D_x)(D_t^{m-1-h-h'}A)(t,D_x).
	\end{equation*}
 Note that in computing $\tbf{B}_{h+1}$ in the last line of \eqref{eq:ComputeBDt}, we use the binomial identity $\binom{m-1-h}{m-1-h-k} = \binom{m-1-h}{k}$ and reorder the summation. This completes the proof after relabelling summation indices.
$\blacksquare$ 
 \end{proof}
 Note that by rewriting the lower order terms as in Corollary \ref{cor:LOT} we clearly see that $B(t,D_t,D_x)$ is of order $m-2$ in $D_t$ rather than of order $m-1$.
As explanatory examples we give a closer look to the operator $B(t,D_t,D_x)$ in the cases $m=2$ and $m=3$.

\begin{example}
	Consider $m=2$: The sum in \eqref{eq:GoodLOTdef} has only one term. We have \begin{equation*}
		\tbf{B}_1(t,D_x) = \tbf{A}_{0}(t,D_x)(D_tA)(t,D_x)
	\end{equation*} with $\tbf{A}_{0}(t,D_x) = \sigma_{0}^{(2)}(\lambda)A^0(t,D_x) = I_2$ (see Lemma \ref{lem:AdjunctRep}).  
	\end{example}

\begin{example} Consider $m=3$. The sum in \eqref{eq:GoodLOTdef} has two terms. We have \begin{equation*}
	\begin{aligned}
	\tbf{B}_1(t,D_x) &= \sum_{h'=0}^1 \binom{2-h'}{0}\tbf{A}_{h'}(t,D_x)(D_t^{2-h'}A)(t,D_x),  \\
	&= \tbf{A}_{0}(t,D_x)(D_t^2A)(t,D_x) + \tbf{A}_1(t,D_x)(D_t A)(t,D_x), \\
	&= (D_t^2 A)(t,D_x) + (A(t,D_x)-\tr(A)(t,D_x)I_3)(D_tA)(t,D_x), 
	\end{aligned} \end{equation*} and \[ 
	\tbf{B}_2(t,D_x) = 2 \tbf{A}_0(t,D_x) (D_tA)(t,D_x) = 2 (D_tA)(t,D_x).
	\] Here  we used the fact that $\tbf{A}_{0}(t,D_x) = \sigma_{0}^{(3)}(\lambda)A^0(t,D_x) = I_3$ and $\sigma_1^{(3)}(\lambda) = -\tr(A)(t,D_x)$ (see Lemma \ref{lem:AdjunctRep}). 
\end{example}

Corollary \ref{cor:LOT} completes Step $2$ of our proof and allows us to transform \eqref{CauchyP} into
 \begin{equation}
 \label{system_m}
 \begin{aligned}
	\adj(I_mD_t - A(t,D_x))(I_mD_t - &A(t,D_x))u \\ 
	&= \delta(t,D_t,D_x)I_mu + B(t,D_t,D_x)u=0,
	\end{aligned}
\end{equation} where $\delta(t,D_t,D_x)$ has symbol $\det(I_m\tau - A(t,\xi))$ and $B(t,D_t,D_x)$ is given by \eqref{eq:GoodLOTdef}. Note that $\delta(t,D_t,D_x)$ is the scalar operator \begin{equation*}
	D_t^m + \sum_{h=1}^{m} c_{h}(t,D_x)D_t^{m-h},
\end{equation*} with $c_{h}(t,\xi)$ homogeneous polynomial of order $h$ with respect to $\xi$ and therefore $\delta(t,D_t,D_x)I_m$ is a decoupled system of $m$ identical scalar differential operators of order $m$ while $B(t,D_t,D_x)$ is a system of differential operators of order $m-1$. As mentioned before, the $c_h(t,\xi)$ are the coefficients of the characteristic polynomial of $A(t,\xi)$, see Appendix \ref{sec:Appendix}.

\subsection{Step 3: Reduction to a first order system of pseudodifferential equations} \label{sec:Reduction}

We now transform the system in \eqref{system_m} into a system of pseudodifferential equations by following Taylor in \cite{Taylor:81}. More precisely, we transform each $m$-th order scalar equation in  $\delta(t,D_t,D_x)I_m$ into a first order pseudodifferential system in Sylvester form. In this way we obtain $m$ systems with identical Sylvester matrix which can be put together in block-diagonal form obtaining a block-diagonal $m^2\times m^2$ matrix with $m$ identical Sylvester blocks. The precise structure of the lower order terms will be worked out in the next subsection. To carry out this transformation, we set \begin{equation}
\label{trans_T}
\begin{aligned}
U &= (U_1,\dots,U_m)^T \in \R^{m^2} \\
U_i &:= \left(D_t^{j-1}\Dx^{m-j}u_i \right)_{j=1,\dots,m} \in \R^m, \quad i=1,\dots,m,
\end{aligned}
\end{equation} where the $u_i$ are the components of the original vector $u$ in \eqref{CauchyP}. We can rewrite the Cauchy problem for \eqref{system_m} as  
\begin{equation} \label{eq:TrafoSys}
	\left\{ \begin{array}{l}
	D_t U = \pazocal A(t,D_x) U + \pazocal B(t,D_x)U, \\
	\left. U \right|_{t=0} = U_0=(U_{0,1}, \cdots, U_{0,m})^T,
	\end{array} \right.
\end{equation}
 where the components $U_{0,i}$  of  the $m^2$-column vector $U_0$ are given by
 \[
 U_{0,i}=\left(D_t^{j-1}\Dx^{m-j}u_{i}(0,x)\right)_{j=1,\cdots, m},
 \]
and $u$ is the solution of the Cauchy problem \eqref{CauchyP} with $u(0,x)=u_0$. Passing now to analyse the matrices $\pazocal A(t,D_x)$ and $\pazocal{B}(t,D_x)$, we have that $\pazocal{A}(t,D_x)$ is an $m^2 \times m^2$ block diagonal matrix of $m$ identical blocks of size $m \times m$ of the type
 \begin{equation} \label{eq:StructurePP}
	\Dx \begin{pmatrix}
	0 & 1 & 0 & \cdots & 0 \\
	0 & 0 & 1 & \cdots & 0 \\
	\vdots & \vdots & \vdots & \ddots & 0 \\
	\vdots & \vdots & \vdots & \cdots & 1 \\
	-c_{m}(t,D_x)\Dx^{-m} & -c_{m-1}(t,D_x)\Dx^{-m+1} & \dots & \dots & -c_1(t,D_x)\Dx^{-1}
	\end{pmatrix}.
\end{equation} 
and the matrix $\pazocal B(t,D_x)$ is composed of $m$ matrices of size $m \times m^2$ as follows:
 \begin{equation} \label{eq:LOTFOS}
	\begin{pmatrix}
	0 &0 & 0 & \dots & 0 & 0 \\
	0 &0 & 0 & \dots & 0 & 0 \\
	\vdots & \vdots & \vdots & \vdots & \vdots & \vdots \\
	l_{i,1}(t,D_x) & l_{i,2}(t,D_x) & \hdots & \hdots & l_{i,m^2-1}(t,D_x) & l_{i,m^2}(t,D_x),
	\end{pmatrix}
\end{equation} $i=1,\dots,m$. Note that the entries of the matrices $\pazocal{A}(t,D_x)$ and $\pazocal{B}(t,D_x)$ are pseudodifferential operators of order $1$ and $0$, respectively.

\subsection{Step 4: Structure of the matrix $\pazocal{B}(t,D_x)$ of the lower order terms} \label{sec:StructBtxi}

To analyse the structure of the $m^2\times m^2$ matrix  $\pac B(t,D_x)$ we recall that it is obtained from  the $m\times m$ matrix $B(t,D_t,D_x)$ in \eqref{system_m} via the transformation \eqref{trans_T}.

From Corollary \ref{cor:LOT} we have that
  \begin{equation} \label{eq:AuxLOT}
	B(t,D_t,D_x)u = \left(- \sum_{h=0}^{m-2} \sum_{j=1}^m b_{ij}^{(h+1)}(t,D_x)D_t^h u_j \right)_{i=1,\dots,m},
\end{equation} where the $b_{ij}^{(h+1)}(t,D_x)$ denote the $(i,j)$-element of $\tbf{B}_{h+1}(t,D_x)$ in \eqref{eq:GoodLOTdef}. By the previously described transform \eqref{trans_T}, we obtain that \begin{eqnarray*}
	D_t^m u_i = -\sum_{h=0}^{m-1} c_{m-h}(t,D_x)D_t^h u_i + \sum_{j=1}^{m} \sum_{h=0}^{m-2} b^{(h+1)}_{ij}(t,D_x)D_t^h u_j
\end{eqnarray*} and, thus, see that the coefficients $b_{ij}^{(1)}(t,D_x)$ in \eqref{eq:AuxLOT} will be associated to $l_{i,1+(j-1)m}(t,D_x)$ for $j=1,\dots,m$, the coefficients $b_{ij}^{(2)}(t,D_x)$ to $l_{i,2+(j-1)m}(t,D_x)$ for $j=1,\dots,m$ and so forth. In particular, we get that $l_{i,m+(j-1)m}(t,D_x) \equiv 0$ for $j=1,\dots,m$ which is due to the fact that \eqref{CauchyP} is homogeneous. As a general formula for the non-zero elements of $\pac B(t,D_x)$, we can write
\begin{equation} \label{l_ij}
	l_{i,h+1+(j-1)m}(t,D_x) = b_{ij}^{(h+1)}(t,D_x)\Dx^{1-m+h} 
\end{equation} for $j=1,\dots,m$ and $h=0,\dots,m-2$. \\

To avoid further complication of the notation, we consider the $b_{ij}^{(l)}(t,\xi)$ from now on as the by $\lara{\xi}^{l-m}$ scaled elements in \eqref{l_ij} if referenced as elements of $\pac B(t,\xi)$. \\

For the convenience of the reader, we conclude this section by illustrating the Steps 1-4 in the case $m=2$ and $m=3$. For simplicity, we take $x\in\R$.

\subsection{Steps 1--4 for $m=2$} \label{sec:ExCase2}

We consider the system \begin{equation} \label{eq:ExCPtbt}
D_t u - A(t)D_x  u = D_t \begin{pmatrix}
u_1 \\ u_2
\end{pmatrix} - \begin{pmatrix}
a_{11}(t) & a_{12}(t) \\
a_{21}(t) & a_{22}(t)
\end{pmatrix} D_x \begin{pmatrix}
u_1 \\ u_2
\end{pmatrix} = 0
\end{equation} for $(t,x) \in [0,T] \times \R$. Computing the adjunct of $I_2 \tau - A(t)\xi$  we obtain \begin{equation*}
\adj(I_2 \tau - A(t)\xi) = \begin{pmatrix}
\tau & 0 \\
0 & \tau
\end{pmatrix} - \begin{pmatrix}
a_{22}(t) & -a_{12}(t) \\
-a_{21}(t) & a_{11}(t)
\end{pmatrix} \xi = I_2 \tau - \adj(A)(t)\xi.
\end{equation*}
Applying the corresponding operator to \eqref{eq:ExCPtbt}, we obtain
 \begin{eqnarray}
	\nonumber \left(I_2 D_t - \adj(A)D_x \right) \left( I_2 D_t - A(t)D_x u\right) &=& \delta(t,D_t,D_x)u - (D_tA)(t)D_x u \\
	\label{eq:TrafoExample1} &=& \delta(t,D_t,D_x)u - \tbf{B}_1(t,D_x)u,
\end{eqnarray} where $\tbf{B}_1(t,D_x)$ is given by \eqref{eq:GoodLOTdef2} with $h=0$. 

Now we set 
\begin{eqnarray*}
U &=& (U_1 , U_2, U_3, U_4)^T = (\Dx u_1, D_t u_1 , \Dx u_2, D_t u_2)^T \\
D_t U &=& ( \Dx U_2, D_t^2 u_1 , \Dx U_4, D_t^2 u_2)^T.
\end{eqnarray*} and, thus, get the system \begin{eqnarray*}
D_t U = \pazocal A(t,D_x) U + \pazocal B(t,D_x) U,
\end{eqnarray*} where $\pac A(t,D_x)$ is a $4 \times 4$ block diagonal matrix, as in \eqref{eq:StructurePP}, with the block  \begin{equation*}
\Dx \begin{pmatrix}
0 & 1 \\
-{\rm det}(A)(t)D_x^2\Dx^{-2} & \tr(A)(t)D_x\Dx^{-1}
\end{pmatrix} \end{equation*} and $\pac B(t,D_x)$ is a $4 \times 4$ matrix of  two $2 \times 4$ blocks \begin{equation*}
\pazocal{B}_i(t,D_x)=
\begin{pmatrix}
0 & 0 & 0 & 0 \\
D_t a_{1i}(t)D_x\Dx^{-1} & 0 & D_t a_{2i}(t)D_x\Dx^{-1} & 0
\end{pmatrix}, \quad i=1,2.
\end{equation*} Note that the entries of the matrix $\pazocal{B}_i(t,D_x)$ can be obtained from \eqref{l_ij} by setting $h=0$ and $j=1,2$.

\subsection{Steps 1--4 for $m=3$} \label{sec:ExCase3}

We consider \begin{equation*}
D_t \begin{pmatrix}
u_1 \\ u_2 \\ u_3
\end{pmatrix} - \begin{pmatrix}
a_{11}(t) & a_{12}(t) & a_{13}(t) \\
a_{21}(t) & a_{22}(t) & a_{23}(t) \\
a_{31}(t) & a_{32}(t) & a_{33}(t)
\end{pmatrix} D_x \begin{pmatrix}
u_1 \\ u_2 \\ u_3
\end{pmatrix} = 0
\end{equation*} for $(t,x) \in [0,T] \times \R$. We have \begin{equation*}
\adj(I_3 \tau - A(t)\xi) = I_3 \tau^2 + (A(t)-\tr(A)(t)I_3)\xi \tau + \adj(A)(t)\xi^2
\end{equation*} and therefore \begin{equation*}
\adj(I_3 D_t - A(t)D_x) = I_3 D_t^2 + (A(t)-\tr(A)(t))I_3 D_tD_x + \adj(A)(t)D_x^2.
\end{equation*} Applying this operator to the original system, we obtain \begin{eqnarray*}
	\adj(I_3 D_t - A(t)D_x)(I_3 D_t - A(t)D_x)u = \delta(t,D_t,D_x)u + B(t,D_t,D_x)u,
\end{eqnarray*} where we used the fact that $\adj(A) = A^2 + c_1A + c_2I_3$ (see example Example \ref{ex:Adjm3}) and set \begin{equation} \label{eq:AuxCompare1} \begin{aligned}
B(t,D_t,D_x) &= - (D_t^2 A)(t)D_x - 2(D_t A)(t)D_xD_t + \tr(A)(t)(D_t A)(t)D_x^2 \\
& \qquad - A(t)(D_t A)(t)D_x^2, \\
&= -\tbf{B}_1(t,D_x) - \tbf{B}_2(t,D_x)D_t,
\end{aligned}
\end{equation} corresponding to \eqref{eq:GoodLOTdef}. Now we introduce \begin{eqnarray*}
	&U &= (U_1, U_2, U_3)^T \in \R^9 \,\,\text{with} \\
	&U_j &= ( \Dx^2 u_j, D_t \Dx u_j, D_t^2 u_j),\quad j=1,2,3.
\end{eqnarray*} Thus, we obtain \begin{equation*}
D_t U = \pazocal A(t,D_x) U + \pazocal B(t,D_x)U,
\end{equation*} where $\pac A(t,D_x)$ is a block diagonal matrix with three blocks of the type

 \begin{equation*}
\Dx \begin{pmatrix}
0 & 1 & 0 \\
0 & 0 & 1 \\
-c_3(t,D_x)\lara{D_x}^{-3} & -c_2(t,D_x) \lara{D_x}^{-2}& -c_1(t,D_x)\lara{D_x}^{-1}
\end{pmatrix}. \end{equation*} By direct computation (see Appendix \ref{sec:Appendix}), we get that $c_{h}(t,D_x) = \sigma_h^{(3)}(\lambda)$, where \begin{eqnarray*}
\sigma_1^{(3)}(\lambda) &=& - \tr(A)(t,D_x) \\
\sigma_2^{(3)}(\lambda) 
&=& a_{11}(t)a_{22}(t)D_x^2 + a_{11}(t)a_{33}(t)D_x^2 + a_{22}(t)a_{33}(t)D_x^2 \\
&& \quad - a_{23}(t)a_{32}(t)D_x^2 - a_{12}(t)a_{21}(t)D_x^2 - a_{31}(t)a_{13}(t)D_x^2\\
\sigma_3^{(3)}(\lambda) &=& -\det(A)(t,D_x).
\end{eqnarray*} Indeed, since
\begin{equation*} \label{eq:Vieta}
	\det(I_3\tau - A) = \prod_{h=1}^3 (\tau-\lambda_i) =\sum_{h=0}^3 \sigma^{(3)}_h(\lambda) \tau^{3-h},
\end{equation*}
it follows that  
\begin{eqnarray*}
	&& \det(I_3\tau - A) = \tau^3 + \underbrace{(-a_{11}-a_{22}-a_{33})}_{\sigma_{1}^{(3)}(\lambda) = -\tr(A)} \tau^2 \\
	&& \quad + \underbrace{(a_{11}a_{22}-a_{12}a_{21}+a_{11}a_{33}-a_{13}a_{31}+a_{22}a_{33}-a_{23}a_{32})}_{\sigma_2^{(3)}(\lambda)} \tau \\
	&& \quad + \underbrace{(-a_{11}a_{22}a_{33}+a_{11}a_{23}a_{32}+a_{12}a_{21}a_{33} - a_{11}a_{23}a_{31} - a_{13}a_{21}a_{32} + a_{13}a_{22}a_{31})}_{\sigma_3^{(3)}(\lambda) = -\det(A)}.
\end{eqnarray*}

Finally, the matrix $\pac B(t,D_x)$ is made of three blocks of $3 \times 9$ matrices \begin{equation*}
\pac B_k(t,D_x)=
\begin{pmatrix}
0 & 0 & 0 & 0 & 0 & 0 & 0 & 0 & 0 \\
0 & 0 & 0 & 0 & 0 & 0 & 0 & 0 & 0 \\
b_{k1}^{(1)} & b_{k1}^{(2)} & 0 & b_{k2}^{(1)} & b_{k2}^{(2)} & 0 & b_{k3}^{(1)} & b_{k3}^{(2)} & 0
\end{pmatrix},
\end{equation*} $k=1,2,3$ which correspond to \eqref{eq:LOTFOS} via formula \eqref{l_ij}. More precisely, we get
\begin{equation}
\label{form_b_3}
\begin{split}
	b^{(1)}_{kj} &= (D_t^2 a_{kj} + 2D_t a_{kj} - \tr(A_0)D_t a_{kj})D_x\lara{D_x}^{-1},\\
 b^{(2)}_{kj}& = (a_{k1}D_t a_{1j} + a_{k2}D_t a_{2j} + a_{k3}D_t a_{3j})D_x^2\lara{D_x}^{-2},
 \end{split}
\end{equation}
for $k=1,2,3$ and $j=1,2$. The elements $b^{(1)}_{kj}$ and $b^{(2)}_{kj}$ can are the scaled $(k,j)$-elements of the matrices $\tbf{B}_1(t,D_x)$ and $\tbf{B}_2(t,D_x)$ from \eqref{eq:AuxCompare1} respectively.

\section{Energy estimate}
\label{sec_energy}

Now we apply the Fourier transform with respect to $x$ to the Cauchy problem in  \eqref{eq:TrafoSys} and set $\pazocal F_{x \rightarrow \xi}(U)(t,\xi) =: V(t,\xi)$. We then obtain 

 \begin{equation} \label{eq:FTSysTrafo} \left\{ \begin{aligned}[l]
& D_t V = \pazocal A(t,\xi) V + \pazocal B(t,\xi) V, \\
& \left. V\right|_{t=0} = V_0,
\end{aligned} \right. \end{equation} 
where $V_0=\widehat{U_0}.$ From now on, we will concentrate on \eqref{eq:FTSysTrafo} and the matrix 

\begin{equation*}
	\pazocal A_0(t,\xi) := \langle \xi \rangle^{-1} \pazocal A(t,\xi).
\end{equation*} 
Note that by construction of $\pazocal A(t,\xi)$, the matrix $\pazocal A_0(t,\xi)$ is made of $m$ identical Sylvester type blocks with eigenvalues $\lambda_l(t,\xi)$, $l=1,\dots,m$, where $\lambda_l(t,\xi)\lara{\xi}$, $l=1,\dots, m$ are the rescaled eigenvalues of the original matrix $A(t,\xi)$ in \eqref{CauchyP}.

\subsection{Step 5: Computing the energy estimate}

Let $\pac{Q}_\eps^{(m)}(t,\xi)$ be the quasi-symmetriser of the matrix $\pazocal A_0(t,\xi)$. By Remark \ref{rem:QuasiSym} it will be a $m^2\times m^2$ block diagonal matrix with $m$ identical blocks given by the quasi-symmetriser $Q_\eps^{(m)}(t,\xi)$ of the defining block of $\pazocal A_0(t,\xi)$ (see Section \ref{SEC:qs} for definition and properties). Hence, we define the energy
 \begin{equation*}
	E_\eps(V)(t,\xi) = \big( \pac{Q}_\eps^{(m)}(t,\xi)V(t,\xi)| V(t,\xi) \big)
\end{equation*} where $(\cdot | \cdot)$ denotes the scalar product in $\R^{m^2}$. To improve the readability, we drop the dependencies on $t$ and $\xi$ in the following unless we find it important to stress. By direct computations we have \begin{eqnarray*}
\partial_t E_\eps &=&(\partial_t\pac{Q}^{(m)}_\eps V | V)+ i(\pac{Q}^{(m)}_\eps D_tV| V)-i(\pac{Q}^{(m)}_\eps V | D_tV)\\
&=&(\partial_t\pac{Q}^{(m)}_\eps V | V)+i(\pac{Q}^{(m)}_\eps(\pac A  V+\pac BV) | V)-i(\pac{Q}^{(m)}_\eps V | \pac A V+ \pac BV)\\
&=&(\partial_t\pac{Q}^{(m)}_\eps V | V)+i\lara{\xi}((\pac{Q}^{(m)}_\eps \pac A_0- \pac A_0^\ast \pac{Q}^{(m)}_\eps)V | V)\\
&& \quad +i((\pac{Q}^{(m)}_\eps \pac B- \pac B^\ast \pac{Q}^{(m)}_\eps)V | V).
\end{eqnarray*}

It follows that \begin{equation} \label{eq:EE} \begin{aligned}
\partial_t E_\eps &\le \frac{|(\partial_t\pac{Q}^{(m)}_\eps V | V)|E_\eps}{(\pac{Q}^{(m)}_\eps V | V)}+|\lara{\xi}((\pac{Q}^{(m)}_\eps \pac A_0- \pac A_0^\ast \pac{Q}^{(m)}_\eps)V | V)| \\
& \quad +|((\pac{Q}^{(m)}_\eps \pac B- \pac B^\ast \pac{Q}^{(m)}_\eps)V | V)|.
\end{aligned} \end{equation}

By Proposition \ref{prop_qs} it follows that $\pac{Q}_\eps^{(m)}(t,\xi)$ is a family of $C^\infty$, non-negative Hermitian matrices such that \begin{equation*} \label{eq:p1}
\pac{Q}_\eps^{(m)}(t,\xi)=\pac{Q}_0^{(m)}(t,\xi)+\eps^2 \pac{Q}_1^{(m)}(t,\xi)+...+\eps^{2(m-1)}\pac{Q}_{m-1}^{(m)}(t,\xi).
\end{equation*} In addition, by the same proposition, there exists a constant $C_m>0$ such that for all $t\in[0,T]$, $\xi\in\R^n$ and $\eps\in(0,1]$ the following estimates hold uniformly in $V \in \R^{m^2}$: \begin{eqnarray}
	\label{eq:p2} & C_m^{-1}\eps^{2(m-1)}|V|^2\le (\pac{Q}^{(m)}_\eps V|V)\le C_m|V|^2,\\
	\label{eq:p3} & |((\pac{Q}_\eps^{(m)}\pac A_0-\pac A_0^\ast \pac{Q}_\eps^{(m)}) V|V)|\le C_m\eps (\pac{Q}_\eps^{(m)} V|V)
\end{eqnarray} 
Finally, the hypothesis \eqref{eq:CondKinSp} on the eigenvalues and Proposition \ref{prop_SM} ensure that the family
$$\{ \pac{Q}_\eps^{(m)}(t,\xi):\, \eps\in(0,1],\, t\in[0,T],\, \xi\in\R^n\}$$ is nearly diagonal.

Note that since the entries of the matrix $A(t,\xi)$ in \eqref{CauchyP} are $C^\infty$ with respect to $t$, the matrices $\pazocal A(t,\xi)$ and $\pazocal B(t,\xi)$ as well as the quasi-symmetriser  have the same regularity properties.

We now proceed by estimating the three summands in the right-hand side of  \eqref{eq:EE}. Due to the block diagonal structure of the matrices involved we can make use of the proof strategy adopted for the scalar case in \cite[Subsections 4.1, 4.2, 4.3]{GR}.

\subsection{First term}
\label{first_term}

Let $k\ge 1$. We write $\frac{|(\partial_t\pac{Q}^{(m)}_\eps V|V)|}{(\pac{Q}^{(m)}_\eps V| V)}$ as
\[
\frac{|(\partial_t\pac{Q}^{(m)}_\eps V,V)|}{(\pac{Q}^{(m)}_\eps V | V)^{1-1/k}(\pac{Q}^{(m)}_\eps V, V)^{1/k}}.
\]
From \eqref{eq:p2} we have \begin{eqnarray*}
\frac{|(\partial_t\pac{Q}^{(m)}_\eps V|V)|}{(\pac{Q}^{(m)}_\eps V| V)}
&\le& \frac{|(\partial_t\pac{Q}^{(m)}_\eps V|V)|}{(\pac{Q}^{(m)}_\eps V | V)^{1-1/k}(C_m^{-1}\eps^{2(m-1)}|V|^2)^{1/k}}\\
&\le& C_m^{1/k}\eps^{-2(m-1)/k}\frac{|(\partial_t \pac{Q}^{(m)}_\eps V|V)|}{(\pac{Q}^{(m)}_\eps V | V)^{1-1/k}|V|^{2/k}}.
\end{eqnarray*}
A block-wise application of Lemma \ref{lem_new} yields the estimate
\begin{eqnarray*}
\int_{0}^T\frac{|(\partial_t\pac{Q}^{(m)}_\eps V|V)|}{(\pac{Q}^{(m)}_\eps V| V)}\, dt
&\le& C_m^{1/k}\eps^{-2(m-1)/k}C_T\sup_{\xi\in\R^n}\Vert \pac{Q}_\eps(\cdot,\xi)\Vert^{1/k}_{{C}^k([0,T])}\\
&\le& C_1\eps^{-2(m-1)/k},
\end{eqnarray*}
for all $\eps\in(0,1]$. Setting $\frac{|(\partial_t\pac{Q}^{(m)}_\eps V | V)|}{(\pac{Q}^{(m)}_\eps V | V)} =: K_\eps(t,\xi)$, we can conclude that
\begin{equation*} \label{eq:Ke}
\frac{|(\partial_t \pac{Q}^{(m)}_\eps V|V)|E_\eps}{(Q^{(m)}_\eps V | V)}=K_\eps(t,\xi)E_\eps,
\end{equation*} with \begin{equation*}\label{EQ:Ke2}
\int_{0}^T K_\eps(t,\xi)\, dt\le C_1\eps^{-2(m-1)/k}.
\end{equation*}

\subsection{Second term}
\label{second_term}

From the property \eqref{eq:p3} we immediately have that
\[
|\lara{\xi}((\pac{Q}^{(m)}_\eps \pac A_0- \pac A_0^\ast \pac{Q}^{(m)}_\eps)V|V)|\le C_m\eps\lara{\xi}(\pac{Q}_\eps^{(m)}V|V)\le C_2\eps\lara{\xi}E_\eps.
\]

\subsection{Third term}
\label{third_term}

In this subsection, we treat the third term on the right-hand side of \eqref{eq:EE}.  By Proposition \ref{prop_qs}(iv) and the definition of the matrix $\pazocal B(t,\xi)$ we have that 
\begin{eqnarray*}
&&((\pac{Q}^{(m)}_\eps \pac B- \pac B^\ast \pac{Q}^{(m)}_\eps)V|V) = ((\pac{Q}_0^{(m)}\pac B- \pac B^\ast \pac{Q}_0^{(m)})V|V)\\
&& \qquad \qquad +\eps^2\sum_{i=1}^m((\pac{Q}^{(m-1)}_\eps(\pi_i\lambda)^\sharp \pac B-\pac B^\ast \pac{Q}^{(m-1)}_\eps(\pi_i\lambda)^\sharp)V|V),
\end{eqnarray*} 
with $\pac{Q}^{(m-1)}_\eps(\pi_i\lambda)^\sharp$ block diagonal matrix with $m$ blocks ${Q}^{(m-1)}_\eps(\pi_i\lambda)^\sharp$ as defined in Proposition \ref{prop_qs}(iv). Note that
\[
(\pac{Q}^{(m-1)}_\eps(\pi_i\lambda)^\sharp \pac B - \pac B^\ast \pac{Q}^{(m-1)}_\eps(\pi_i\lambda)^\sharp)= 0,
\]
for all $i=1,\dots,m$, due to the structure of zeros in $\pac B$ and in $\pac{Q}^{(m-1)}_\eps(\pi_i\lambda)^\sharp$. Thus,
 \begin{equation*}
((\pac{Q}^{(m)}_\eps \pac B - \pac B^\ast \pac{Q}^{(m)}_\eps)V|V)=((\pac{Q}_0^{(m)} \pac B - \pac B^\ast \pac{Q}_0^{(m)})V|V).
\end{equation*} 
Since from Proposition \ref{prop_qs}(i) the quasi-symmetriser is made of non-negative matrices we have that 
\[
(\pac{Q}^{(m)}_0 V,V)\le E_\eps.
\]
It is purpose of the next section to find suitable Levi conditions on $\pac B(t,\xi)$ such that 
 \begin{equation} \label{eq:cB}
|((\pac{Q}_0^{(m)} \pac B - \pac B^\ast \pac{Q}_0^{(m)})V | V)|\le C_3 (\pac{Q}^{(m)}_0 V | V)\le C_3 E_\eps
\end{equation}
holds for some constant $C_3>0$ independent of $t\in[0,T]$, $\xi\in\R^n$ and $V\in\C^{m^2}$. We will then formulate these Levi-type conditions in terms of the matrix $A$ in \eqref{CauchyP}.

\section{Estimates for the lower order terms}
\label{sec_lot}

We remind the reader of the fact that the $b_{ij}^{(l)}(t,\xi)$, if referenced as elements of $\pac B(t,\xi)$, are the by $\lara{\xi}^{l-m}$ scaled $(i,j)$-elements of $\tbf{B}_l(t,\xi)$ in \eqref{eq:GoodLOTdef}. See also Section \ref{sec:StructBtxi} for details. \\

To start, we rewrite $((\pac{Q}_0^{(m)} \pac B - \pac B^\ast \pac{Q}_0^{(m)})V | V)$ in terms of the matrix $\pac W^{(m)}$. Recall that from Section \ref{SEC:qs}, $\pac W^{(m)}$ is the $m^2\times m^2$ block diagonal matrix with $m$ identical blocks
\[
W^{(m)}=\begin{pmatrix}
W^{(m)}_1(\lambda)\\
\vdots\\
W^{(m)}_m(\lambda)
\end{pmatrix},
\]
with 
\[
W^{(m)}_i(\lambda)=\big(\sigma_{m-1}^{(m-1)}(\pi_i\lambda),...,\sigma_1^{(m-1)}(\pi_i\lambda),1\big),\quad 1\le i\le m.
\]

From Proposition \ref{prop_qs}(v) we have
\begin{eqnarray*}
((\pac{Q}_0^{(m)} \pac B - \pac B^\ast \pac{Q}_0^{(m)})V | V) &=& (m-1)!((\pac{W}^{(m)} \pac BV | \pac{W}^{(m)}V)-(\pac{W}^{(m)}V |\pac{W}^{(m)}\pac BV))\\
&=& 2i(m-1)!\ima (\pac{W}^{(m)}\pac B V | \pac{W}^{(m)}V).
\end{eqnarray*}
It follows that \begin{equation*}
	|((\pac{Q}_0^{(m)} B-B^\ast \pac{Q}_0^{(m)})V | V)|\le 2(m-1)!|\pac{W}^{(m)}BV||\pac{W}^{(m)}V|.
\end{equation*} Since \begin{equation*}
	(\pac{Q}_0^{(m)} V | V)=(m-1)!|\pac{W}^{(m)}V|^2,
\end{equation*} we have that if \begin{equation} \label{eq:cB2}
|\pac{W}^{(m)}\pac BV| \le C|\pac{W}^{(m)}V|
\end{equation} holds true for some constant $C>0$, independent of $t$, $\xi$ and $V$, then estimate \eqref{eq:cB} will hold true as well. 

In the sequel, for the sake of simplicity we will make use of the following notation: given $f$ and $g$ two real valued functions in the variable $y$, $f(y)\prec g(y)$ if there exists a constant $C>0$ such that $f(y)\le C g(y)$ for all $y$. More precisely, we will set $y=(t,\xi)$ or $y=(t,\xi,V)$. Thus, \eqref{eq:cB2} can be rewritten as 
\[
|\pac{W}^{(m)}\pac BV| \prec|\pac{W}^{(m)}V|.
\]
In analogy to the scalar case in \cite{GR} we will now focus on \eqref{eq:cB2}. Before proceeding with our general result, for advantage of the reader we will illustrate the main ideas leading to the Levi-type conditions on $\pac B$ in the case $m=2$ and $m=3$.

\subsection{The case $m=2$}
\label{lot_2_zone}

For simplicity we take $n=1$. From Subsection \ref{sec:ExCase2} and Subsection \ref{Ex_qs} we have that
\[
\pazocal B(t,\xi)=\begin{pmatrix}
	0 & 0 & 0 & 0 \\
	D_t a_{11}(t) & 0 & D_t a_{21}(t) & 0 \\
	0 & 0 & 0 & 0 \\
	D_t a_{12}(t) & 0 & D_t a_{22}(t) & 0
\end{pmatrix}\xi\lara{\xi}^{-1}
\]
and
\[
\pazocal W^{(2)}(t,\xi)=\begin{pmatrix}
		-\lambda_2 & 1 & 0 & 0 \\
		-\lambda_1 & 1 & 0 & 0 \\
		0 & 0 & -\lambda_2 & 1 \\
		0 & 0 & -\lambda_1 & 1 \\
	\end{pmatrix} ,
\]
respectively. We have \begin{eqnarray*}
	\pazocal W^{(2)} \pazocal B V &=& \begin{pmatrix}
		-\lambda_2 & 1 & 0 & 0 \\
		-\lambda_1 & 1 & 0 & 0 \\
		0 & 0 & -\lambda_1 & 1 \\
		0 & 0 & -\lambda_2 & 1 \\
	\end{pmatrix} \begin{pmatrix}
	0 & 0 & 0 & 0 \\
	D_t a_{11}(t) & 0 & D_t a_{21}(t) & 0 \\
	0 & 0 & 0 & 0 \\
	D_t a_{12}(t) & 0 & D_t a_{22}(t) & 0
\end{pmatrix} \xi\lara{\xi}^{-1}\begin{pmatrix}
V_1 \\ V_2 \\ V_3 \\ V_4
\end{pmatrix} \\
&=& \begin{pmatrix}
	D_t a_{11}(t) & 0 & D_t a_{21}(t) & 0 \\
	D_t a_{11}(t) & 0 & D_t a_{21}(t) & 0 \\
	D_t a_{12}(t) & 0 & D_t a_{22}(t) & 0 \\
	D_t a_{12}(t) & 0 & D_t a_{22}(t) & 0
\end{pmatrix} \begin{pmatrix}
V_1 \\ V_2 \\ V_3 \\ V_4
\end{pmatrix}\xi\lara{\xi}^{-1}
= \begin{pmatrix}
D_t a_{11}(t)V_1 + D_t a_{21}(t) V_3 \\
D_t a_{11}(t)V_1 + D_t a_{21}(t) V_3 \\
D_t a_{12}(t)V_1 + D_t a_{22}(t) V_3 \\
D_t a_{12}(t)V_1 + D_t a_{22}(t) V_3
\end{pmatrix}\xi\lara{\xi}^{-1}
\end{eqnarray*} and \begin{equation*}
\pazocal W ^{(2)}V = \begin{pmatrix}
-\lambda_2 & 1 & 0 & 0 \\
-\lambda_1 & 1 & 0 & 0 \\
0 & 0 & -\lambda_2 & 1 \\
0 & 0 & -\lambda_1 & 1 \\
\end{pmatrix} \begin{pmatrix}
V_1 \\ V_2 \\ V_3 \\ V_4
\end{pmatrix}  = \begin{pmatrix}
-\lambda_2 V_1 + V_2 \\
-\lambda_1 V_1 + V_2 \\
-\lambda_2 V_3 + V_4 \\
-\lambda_1 V_3 + V_4
\end{pmatrix}. 
\end{equation*}
Thus, we obtain that $|\pazocal W^{(2)} \pazocal B V|^ 2\prec |\pazocal W^{(2)}V|^2$ is equivalent to 
\begin{equation} 
\label{est_case2}\begin{aligned}
& |D_t a_{11}(t)V_1 + D_t a_{21}(t) V_3|^2\xi\lara{\xi}^{-1}
+ |D_t a_{12}(t)V_1 + D_t a_{22}(t) V_3|^2\xi\lara{\xi}^{-1} \\
& \quad\prec |-\lambda_2 V_1 + V_2|^2 +
|-\lambda_1 V_1 + V_2|^2 + |-\lambda_2 V_3 + V_4|^2
+ |-\lambda_1 V_3 + V_4|^2.
\end{aligned} \end{equation} 
We now estimate the left-hand side of \eqref{est_case2} from above and the right-hand side from below. We get
\begin{equation*} \label{eq:EstLC1}
\begin{aligned}
& |D_t a_{11}(t)V_1 + D_t a_{21}(t) V_3|^2
+ |D_t a_{12}(t)V_1 + D_t a_{22}(t) V_3|^2 \\
& \quad\prec \left( |D_t a_{11}(t)|^2 + |D_t a_{12}(t)|^2 \right)|V_1|^2 + \left( |D_t a_{21}(t)|^2 + |D_t a_{22}(t)|^2 \right) |V_3|^2
\end{aligned}	
\end{equation*} and, by using the inequality $|z_1|^2 + |z_2|^2 \ge \frac{1}{2} |z_1 - z_2|^2$, $z_1,z_2\in \C$, and the condition \eqref{eq:CondKinSp} on the eigenvalues,\begin{equation*} \label{eq:EstLC2}
\begin{aligned}
& |-\lambda_2 V_1 + V_2|^2 +
|-\lambda_1 V_1 + V_2|^2 + |-\lambda_2 V_3 + V_4|^2
+ |-\lambda_1 V_3 + V_4|^2 \\
& \quad \succ (\lambda_2-\lambda_1)^2|V_1|^2 + (\lambda_2-\lambda_1)^2|V_3|^2 \\
& \quad \succ (\lambda^2_1 + \lambda_2^2)|V_1|^2 + (\lambda^2_1 + \lambda_2^2)|V_2|^2.
\end{aligned}
\end{equation*} 
Combining the last two inequalities, we finally obtain that $|\pazocal W^{(2)} \pazocal B V|^2\prec |\pazocal W V|^2$ provided that \begin{equation} \label{cond:LevytypeEx} \begin{aligned}
	&(|D_t a_{11}(t)|^2 + |D_t a_{21}(t)|^2)\xi\lara{\xi}^{-1} \prec  \lambda_1^2(t,\xi) + \lambda_2^2(t,\xi), \\
	&(|D_t a_{12}(t)|^2 + |D_t a_{22}(t)|^2)\xi\lara{\xi}^{-1} \prec \lambda_1^2(t,\xi) + \lambda_2^2(t,\xi).
\end{aligned} \end{equation}

This is a Levi-type condition on the matrix of the lower order terms $\pazocal B$ written in terms of the entries of the original matrix $A$ in \eqref{CauchyP}. Note that by adopting the notations introduced in Subsection \ref{sec:ExCase3} for the matrix $\pazocal B$ in the case $m=2$ as well, i.e., 
\[
\pazocal B=\begin{pmatrix}
0 & 0 & 0 & 0\\
b_{11}^{(1)}(t) & 0 & b_{12}^{(1)}(t) &0 \\
0 & 0 & 0 & 0\\
b_{21}^{(1)}(t) & 0 & b_{22}^{(1)}(t) & 0
\end{pmatrix}
\]
the Levi-type conditions above can be written as
\[
\begin{aligned}
|b_{11}^{(1)}|^2 +|b_{21}^{(1)}|^2 &\prec  \lambda_1^2 +\lambda_2^2\\
|b_{12}^{(1)}|^2 +|b_{22}^{(1)}|^2 &\prec  \lambda_1^2 +\lambda_2^2,
\end{aligned}
\]
where $\lambda_1^2+\lambda_2^2$ is the entry $q_{11}$ of the symmetriser of the matrix $A_0=A\lara{\xi}^{-1}$.

\subsection{The case $m=3$}
\label{lot_3_zone}

We begin by recalling that from Subsection \ref{sec:ExCase3} the $9\times 9$ matrix $\pazocal B(t,\xi)$ is given by the $3\times 9$ matrices $\pazocal B_k(t,\xi)$, $k=1,2,3$, as follows:
\[
\pazocal B=\begin{pmatrix}
\pazocal B_1\\
\pazocal B_2\\
\pazocal B_3
\end{pmatrix}
=\begin{pmatrix}
0 & 0 & 0 & 0 & 0 & 0 & 0 & 0 & 0 \\
0 & 0 & 0 & 0 & 0 & 0 & 0 & 0 & 0 \\
b_{11}^{(1)}(t) & b_{11}^{(2)}(t) & 0 & b_{12}^{(1)}(t) & b_{12}^{(2)}(t) & 0 & b_{13}^{(1)}(t) & b_{13}^{(2)}(t) & 0\\
0 & 0 & 0 & 0 & 0 & 0 & 0 & 0 & 0 \\
0 & 0 & 0 & 0 & 0 & 0 & 0 & 0 & 0 \\
b_{21}^{(1)}(t) & b_{21}^{(2)}(t) & 0 & b_{22}^{(1)}(t) & b_{22}^{(2)}(t) & 0 & b_{23}^{(1)}(t) & b_{23}^{(2)}(t) & 0\\
0 & 0 & 0 & 0 & 0 & 0 & 0 & 0 & 0 \\
0 & 0 & 0 & 0 & 0 & 0 & 0 & 0 & 0 \\
b_{31}^{(1)}(t) & b_{31}^{(2)}(t) & 0 & b_{32}^{(1)}(t) & b_{32}^{(2)}(t) & 0 & b_{33}^{(1)}(t) & b_{33}^{(2)}(t) & 0
\end{pmatrix}.
\]
Hence,
\begin{equation} \label{eq:aux1}
\pazocal W^{(3)} \pazocal B = \begin{pmatrix}
b_{11}^{(1)} & b_{11}^{(2)} & 0 & b_{12}^{(1)} & b_{12}^{(2)} & 0 & b_{13}^{(1)} & b_{13}^{(2)} & 0 \\
b_{11}^{(1)} & b_{11}^{(2)} & 0 & b_{12}^{(1)} & b_{12}^{(2)} & 0 & b_{13}^{(1)} & b_{13}^{(2)} & 0 \\
b_{11}^{(1)} & b_{11}^{(2)} & 0 & b_{12}^{(1)} & b_{12}^{(2)} & 0 & b_{13}^{(1)} & b_{13}^{(2)} & 0 \\
b_{21}^{(1)} & b_{21}^{(2)} & 0 & b_{22}^{(1)} & b_{22}^{(2)} & 0 & b_{23}^{(1)} & b_{23}^{(2)} & 0 \\
b_{21}^{(1)} & b_{21}^{(2)} & 0 & b_{22}^{(1)} & b_{22}^{(2)} & 0 & b_{23}^{(1)} & b_{23}^{(2)} & 0 \\
b_{21}^{(1)} & b_{21}^{(2)} & 0 & b_{22}^{(1)} & b_{22}^{(2)} & 0 & b_{23}^{(1)} & b_{23}^{(2)} & 0 \\
b_{31}^{(1)} & b_{31}^{(2)} & 0 & b_{32}^{(1)} & b_{32}^{(2)} & 0 & b_{33}^{(1)} & b_{33}^{(2)} & 0 \\
b_{31}^{(1)} & b_{31}^{(2)} & 0 & b_{32}^{(1)} & b_{32}^{(2)} & 0 & b_{33}^{(1)} & b_{33}^{(2)} & 0 \\
b_{31}^{(1)} & b_{31}^{(2)} & 0 & b_{32}^{(1)} & b_{32}^{(2)} & 0 & b_{33}^{(1)} & b_{33}^{(2)} & 0
\end{pmatrix}, \end{equation} and \begin{equation} \label{eq:aux2}
\pazocal W^{(3)} V = \begin{pmatrix}
\lambda_2\lambda_3 V_1 - (\lambda_2 + \lambda_3)V_2 + V_3 \\
\lambda_3\lambda_1 V_1 - (\lambda_3 + \lambda_1)V_2 + V_3 \\
\lambda_1\lambda_2 V_1 - (\lambda_1 + \lambda_2)V_2 + V_3 \\
\lambda_2\lambda_3 V_4 - (\lambda_2 + \lambda_3)V_5 + V_6 \\
\lambda_3\lambda_1 V_4 - (\lambda_3 + \lambda_1)V_5 + V_6 \\
\lambda_1\lambda_2 V_4 - (\lambda_1 + \lambda_2)V_5 + V_6 \\
\lambda_2\lambda_3 V_7 - (\lambda_2 + \lambda_3)V_8 + V_9 \\
\lambda_3\lambda_1 V_7 - (\lambda_3 + \lambda_1)V_8 + V_9 \\
\lambda_1\lambda_2 V_7 - (\lambda_1 + \lambda_2)V_8 + V_9 \\
\end{pmatrix}.
\end{equation}
Note that $\pazocal W^{(3)} \pazocal B$ is a $9\times 9$ matrix with three blocks of three identical rows and  $\pazocal W^{(3)} V$ is a $9\times 1$ matrix with three blocks of rows having the same structure in $\lambda_1$, $\lambda_2$ and $\lambda_3$.

From \eqref{eq:aux1}, we deduce that 
\begin{equation*} \begin{aligned}
& |\pazocal W \pazocal B V|^2\prec \left( |b_{11}^{(1)}|^2 + |b_{21}^{(1)}|^2 + |b_{31}^{(1)}|^2 \right) |V_1|^2 \left( |b_{11}^{(2)}|^2 + |b_{21}^{(2)}|^2 + |b_{31}^{(2)}|^2 \right) |V_2|^2 \\
& \qquad \quad + \left( |b_{12}^{(1)}|^2 + |b_{22}^{(1)}|^2 + |b_{32}^{(1)}|^2 \right) |V_4|^2 + \left( |b_{12}^{(2)}|^2 + |b_{22}^{(2)}|^2 + |b_{32}^{(21)}|^2 \right) |V_5|^2 \\
& \qquad \quad + \left( |b_{13}^{(1)}|^2 + |b_{23}^{(1)}|^2 + |b_{33}^{(1)}|^2 \right) |V_7|^2 + \left( |b_{13}^{(2)}|^2 + |b_{23}^{(2)}|^2 + |b_{33}^{(21)}|^2 \right) |V_8|^2.
\end{aligned} \end{equation*} 
Taking inspiration from the Levi conditions in \cite{GR} and in analogy with the case $m=2$ we set
\begin{equation}\label{eq:condm3}\begin{aligned}
	|b_{11}^{(1)}|^2 +|b_{21}^{(1)}|^2+|b_{31}^{(1)}|^2 &\prec \lambda_1^2 \lambda_2^2 + \lambda_1^2\lambda_3^2 + \lambda_2^2 \lambda_3^2 \\
	|b_{12}^{(1)}|^2 +|b_{22}^{(1)}|^2+|b_{32}^{(1)}|^2 &\prec \lambda_1^2 \lambda_2^2 + \lambda_1^2\lambda_3^2 + \lambda_2^2 \lambda_3^2 \\
	|b_{13}^{(1)}|^2 +|b_{23}^{(1)}|^2+|b_{33}^{(1)}|^2 &\prec \lambda_1^2 \lambda_2^2 + \lambda_1^2\lambda_3^2 + \lambda_2^2 \lambda_3^2 \\
	|b_{11}^{(2)}|^2 +|b_{21}^{(2)}|^2+|b_{31}^{(2)}|^2 &\prec (\lambda_1+\lambda_2)^2 + (\lambda_1+\lambda_3)^2 + (\lambda_2+\lambda_3)^2 \\
	|b_{12}^{(2)}|^2 +|b_{22}^{(2)}|^2+|b_{32}^{(2)}|^2 &\prec (\lambda_1+\lambda_2)^2 + (\lambda_1+\lambda_3)^2 + (\lambda_2+\lambda_3)^2 \\
	|b_{13}^{(2)}|^2 +|b_{23}^{(2)}|^2+|b_{33}^{(2)}|^2 &\prec (\lambda_1+\lambda_2)^2 + (\lambda_1+\lambda_3)^2 + (\lambda_2+\lambda_3)^2.
\end{aligned} \end{equation} 
Note that $\lambda_1^2 \lambda_2^2 + \lambda_1^2\lambda_3^2 + \lambda_2^2 \lambda_3^2 $ and $(\lambda_1+\lambda_2)^2 + (\lambda_1+\lambda_3)^2 + (\lambda_2+\lambda_3)^2$ are the entries $q_{11}$ and $q_{22}$ of the symmetriser of $A_0=\lara{\xi}^{-1}A$, respectively. By imposing these conditions on the lower order terms we have that
\begin{equation} \label{eq:Aux1}
\begin{aligned}
& |\pazocal W^{(3)} \pazocal B V|^2\prec \left( \lambda_1^2 \lambda_2^2 + \lambda_1^2\lambda_3^2 + \lambda_2^2 \lambda_3^2 \right) (|V_1|^2 + |V_4|^2 + |V_7|^2) \\
& \quad + \left( (\lambda_1+\lambda_2)^2 + (\lambda_1+\lambda_3)^2 + (\lambda_2+\lambda_3)^2 \right) (|V_2|^2 + |V_5|^2 + |V_8|^2).
\end{aligned} \end{equation} 
Making a comparison with \cite{GR}, we observe that $V_1$, $V_4$, and $V_7$ play the role of $V_1$ in \cite{GR} and  $V_2$, $V_5$ and $V_8$ play the role of $V_2$ in \cite{GR}. Finally, from \eqref{eq:aux2}, we obtain that 
\begin{eqnarray*}
 |\pazocal W^{(3)} V|^2 &=& |\lambda_2\lambda_3 V_1 - (\lambda_2 + \lambda_3)V_2 + V_3|^2 + |\lambda_3\lambda_1 V_1 - (\lambda_3 + \lambda_1)V_2 + V_3|^2 \\
	&& \quad + |\lambda_1\lambda_2 V_1 - (\lambda_1 + \lambda_2)V_2 + V_3|^2
	+ |\lambda_2\lambda_3 V_4 - (\lambda_2 + \lambda_3)V_5 + V_6|^2 \\
	&& \quad + |\lambda_3\lambda_1 V_4 - (\lambda_3 + \lambda_1)V_5 + V_6|^2 + |\lambda_1\lambda_2 V_4 - (\lambda_1 + \lambda_2)V_5 + V_6|^2 \\
	&& \quad + |\lambda_2\lambda_3 V_7 - (\lambda_2 + \lambda_3)V_8 + V_9|^2 + |\lambda_3\lambda_1 V_7 - (\lambda_3 + \lambda_1)V_8 + V_9|^2 \\
	&& \quad + |\lambda_1\lambda_2 V_7 - (\lambda_1 + \lambda_2)V_8 + V_9|^2.
\end{eqnarray*} 
It is our aim to prove that $|\pazocal W^{(3)} \pazocal B V|^2\prec|\pazocal W^{(3)} V|^2$. We do this by estimating $|\pazocal W^{(3)} \pazocal B V|^2$ and $|\pazocal W^{(3)} V|^2$ in different zones. More precisely, inspired by \cite{GR} we decompose $\R^9$ as \begin{equation*}
\Sigma_1^{\delta_1} \cup (\Sigma_1^{\delta_1})^c,
\end{equation*} where  \begin{eqnarray*}
	&& \Sigma_1^{\delta_1} := \Big\{ V \in \R^9 : \sum_{1\leq i<j\leq3} (\lambda_i + \lambda_j)^2 (|V_2|^2 + |V_5|^2 + |V_8|^2) \\
	&& \qquad \qquad \qquad \qquad \qquad \leq \delta_1 \sum_{1\leq i<j\leq3} \lambda_i^2\lambda_j^2 (|V_1|^2 + |V_4|^2 + |V_7|^2) \Big\}
\end{eqnarray*} 
for some $\delta_1 >0$.

\paragraph{Estimate on $\Sigma_1^{\delta_1}$.} By definition of the zone, we obtain from \eqref{eq:Aux1} \begin{equation*}
|\pazocal W^{(3)} \pazocal B V|^2\prec \left( \lambda_1^2\lambda_2^2 +\lambda_2^2\lambda_3^2 + \lambda_1^2\lambda_3^2 \right)(|V_1|^2 + |V_4|^2 + |V_7|^2).
\end{equation*} Thanks to the hypothesis \eqref{eq:CondKinSp} on the eigenvalues, we have the following estimates\footnote{Using $|z_1|^2 + |z_2|^2 + |z_3|^2 \ge\frac{1}{2}( |z_1 - z_2|^2 + |z_1-z_3|^2 + |z_2 - z_3|^2)$, $z_1,z_2,z_3\in\C$.} 
\begin{eqnarray*}
	|\pazocal W^{(3)} V|^2 &\succ & |(\lambda_2\lambda_3 - \lambda_3\lambda_1)V_1 - (\lambda_2-\lambda_1)V_2 |^2 \\
	&& \qquad + |(\lambda_2\lambda_3-\lambda_1\lambda_2 )V_1 - (\lambda_3-\lambda_1)V_2|^2\\
	&& \qquad + |(\lambda_3\lambda_1 - \lambda_1\lambda_2)V_1 - (\lambda_3-\lambda_2)V_2|^2 \\
	 &\succ &  (\lambda_1^2 + \lambda_2^2)|\lambda_3V_1-V_2|^2 + (\lambda_3^2 + \lambda_1^2)|\lambda_2V_1-V_2|^2 \\  
	&& \qquad + (\lambda_2^2 + \lambda_3^2)|\lambda_1V_1-V_2|^2 \\
	&\succ & \lambda_1^2|(\lambda_3-\lambda_2)V_1|^2+\lambda_3^2|(\lambda_2-\lambda_1)V_1|^2\\
	&\succ & \left( \lambda_1^2\lambda_2^2 +\lambda_2^2\lambda_3^2 + \lambda_1^2\lambda_3^2 \right)|V_1|^2.
\end{eqnarray*} 
Note that in the previous bound from below we have taken in considerations only the terms with $V_1$, $V_2$ and $V_3$. Repeating the same arguments for the groups of terms with $V_4$, $V_5$, $V_6$ and $V_7$, $V_8$, $V_9$, respectively, we get that 
\[
|\pazocal W^{(3)} V|^2 \succ \left( \lambda_1^2\lambda_2^2 +\lambda_2^2\lambda_3^2 + \lambda_1^2\lambda_3^2 \right)|V_4|^2
\]
and
\[
|\pazocal W^{(3)} V|^2 \succ \left( \lambda_1^2\lambda_2^2 +\lambda_2^2\lambda_3^2 + \lambda_1^2\lambda_3^2 \right)|V_7|^2.
\]
Hence, 
\begin{equation*}
|\pazocal W^{(3)}V|^2 \succ \Big( \sum_{1\leq i<j\leq3} \lambda_i^2\lambda_j^2 \Big) (|V_1|^2 + |V_4|^2 + |V_7|^2).
\end{equation*} 
Thus, combining the last estimate with \eqref{eq:Aux1}, we obtain $|\pazocal W^{(3)} \pazocal B V|\prec |\pazocal W^{(3)} V|$ for all $V \in \Sigma_1^{\delta_1}$. No assumptions have been made on $\delta_1>0$.

\paragraph{Estimate on $(\Sigma^{\delta_1}_1)^c$.} By definition of the zone $(\Sigma^{\delta_1}_1)^c$, we obtain from \eqref{eq:Aux1} that   
\begin{equation} \label{eq:aux4}
|\pazocal W^{(3)} \pazocal B V|^2\prec \big( 1+\frac{1}{\delta_1} \big) \Big( \sum_{1\leq i<j\leq3} (\lambda_i+\lambda_j)^2 \Big) (|V_2|^2 + |V_5|^2 + |V_8|^2).
\end{equation} 
Further, by taking into considerations only the terms with $V_1$, $V_2$ and $V_3$ in $|\pazocal W^{(3)} V|^2$ we have 
\begin{eqnarray}
\nonumber	|\pazocal W^{(3)} V|^2 &=& |\lambda_2\lambda_3 V_1 - (\lambda_2 + \lambda_3)V_2 + V_3|^2 + |\lambda_3\lambda_1 V_1 - (\lambda_3 + \lambda_1)V_2 + V_3|^2 \\
\nonumber && \quad + |\lambda_1\lambda_2 V_1 - (\lambda_1 + \lambda_2)V_2 + V_3|^2 \\
\nonumber	&\succ& \gamma_1 \big( |(\lambda_2+\lambda_3)V_2-V_3|^2 + |(\lambda_3+\lambda_1)V_2-V_3|^2 \\
\label{eq:aux3}	&& \quad + |(\lambda_1+\lambda_2)V_2-V_3|^2 \big) - \gamma_2\big( \lambda_1^2\lambda_2^2 + \lambda_1^2\lambda_3^2 + \lambda_2^2\lambda_3^2 \big)|V_1|^2
\end{eqnarray} 
for some constant $\gamma_1, \gamma_2>0$ suitably chosen\footnote{Using $|z_1-z_2|^2 \geq \gamma_1|z_1|^2 - \gamma_2|z_2|^2$ with $\gamma_1=\frac{1}{2}$, $\gamma_2=1$.}. 
The hypothesis \eqref{eq:CondKinSp} implies
 \begin{eqnarray*}
	&&(\lambda_2-\lambda_1)^2 + (\lambda_3-\lambda_2)^2 + (\lambda_3-\lambda_1)^2 \geq \frac{2}{C} (\lambda_1^2 + \lambda_2^2 + \lambda_3^2) \\
	&& \qquad \geq \frac{1}{2C}\big( (\lambda_1 + \lambda_2)^2 + (\lambda_1 + \lambda_3)^2 + (\lambda_2+\lambda_3)^2 \big).
\end{eqnarray*} Applying the last inequality to \eqref{eq:aux3}, we obtain \begin{eqnarray*}
|\pazocal W^{(3)} V|^2 &\succ& \gamma_1 \big( |(\lambda_2+\lambda_3)V_2-V_3|^2 + |(\lambda_3+\lambda_1)V_2-V_3|^2 \\
&& \quad + |(\lambda_1+\lambda_2)V_2-V_3|^2 \big) - \gamma_2\big( \lambda_1^2\lambda_2^2 + \lambda_1^2\lambda_3^2 + \lambda_2^2\lambda_3^2 \big)|V_1|^2 \\
&\succ& \gamma_1 ((\lambda_2-\lambda_1)^2 + (\lambda_3-\lambda_2)^2 + (\lambda_3-\lambda_1)^2) |V_2|^2 \\
&& \quad - \gamma_2\big( \lambda_1^2\lambda_2^2 + \lambda_1^2\lambda_3^2 + \lambda_2^2\lambda_3^2 \big)|V_1|^2 \\
&\succ& \gamma_1' ( (\lambda_1 + \lambda_2)^2 + (\lambda_1 + \lambda_3)^2 + (\lambda_2+\lambda_3)^2)|V_2|^2 \\
&& \quad - \gamma_2\big( \lambda_1^2\lambda_2^2 + \lambda_1^2\lambda_3^2 + \lambda_2^2\lambda_3^2 \big)|V_1|^2. \\
\end{eqnarray*} 
Now, repeating the same argument for the terms involving $V_4, V_5, V_6$ and $V_7, V_8, V_9$, respectively, we get 
\[
|\pazocal W^{(3)} V|^2\succ \gamma_1' ( (\lambda_1 + \lambda_2)^2 + (\lambda_1 + \lambda_3)^2 + (\lambda_2+\lambda_3)^2)|V_5|^2 
 - \gamma_2\big( \lambda_1^2\lambda_2^2 + \lambda_1^2\lambda_3^2 + \lambda_2^2\lambda_3^2 \big)|V_4|^2
\]
and
\[
|\pazocal W^{(3)} V|^2\succ \gamma_1' ( (\lambda_1 + \lambda_2)^2 + (\lambda_1 + \lambda_3)^2 + (\lambda_2+\lambda_3)^2)|V_8|^2 
 - \gamma_2\big( \lambda_1^2\lambda_2^2 + \lambda_1^2\lambda_3^2 + \lambda_2^2\lambda_3^2 \big)|V_7|^2.
\]
It follows that for all $V\in (\Sigma_1^{\delta_1})^c$ the bound from below
\[
|\pazocal W^{(3)} V|^2 \succ \big(\gamma_1' - \frac{\gamma_2}{\delta_1} \big) \Big( \sum_{1\leq i<j\leq3} (\lambda_i+\lambda_j)^2 \Big) (|V_2|^2 +|V_5|^2 + |V_8|^2)
\]
holds, provided that $\delta_1$ is chosen large enough. Combining this with \eqref{eq:aux4}, we get $|\pazocal W^{(3)} \pazocal B V|\prec |\pazocal W V|$ on $(\Sigma_1^{\delta_1})^c$ and, thus, on $\R^9$.

\subsection{The general case} \label{sec:GeneralCaseB}

Recall from Section \ref{sec:StructBtxi} that the $m^2 \times m^2$ matrix $\pac B(t,\xi)$ is made up of $m$ matrices of dimension $m \times m^2$ that contain only in the last line non-zero elements, see \eqref{eq:LOTFOS}. To not further complicate the notation, we will in what follows denote $\pazocal W^{(m)}$ simply by $\pac W$ and will also assume that the $b_{ij}^{(l)}(t,\xi)$ in $\pac B(t,\xi)$ are properly scaled by $\lara{\xi}^{l-m}$. For that see Section \ref{sec:StructBtxi}, specifically formula \eqref{l_ij}. Thus, we have \begin{equation*}
	\pac B(t,\xi) = \begin{pmatrix}
	\pac B_1(t,\xi) \\ \vdots \\ \pac B_m(t,\xi)
	\end{pmatrix}, \quad B_i(t,\xi) = \begin{pmatrix}
	0 & 0 & \cdots & 0 \\
	\pac B^{(1)}_i(t,\xi)& \pac B^{(2)}_i(t,\xi) & \cdots & \pac B^{(m)}_i(t,\xi)
	\end{pmatrix}.
\end{equation*} The $\pac B_i(t,\xi)$ are then given by \begin{equation*}
	\pac B_i(t,\xi) = \left( b_{ij}^{(1)}(t,\xi), b_{ij}^{(2)}(t,\xi), \cdots, b_{ij}^{(m-1)}(t,\xi), 0 \right)
\end{equation*} for $1 \leq i \leq m$. Thus, we obtain \begin{equation} \label{eq:MatrixScale}
\pac W \pac B = \begin{pmatrix}
b_{11}^{(1)} & \cdots & b_{11}^{(m-1)} & 0 & \cdots & b_{1m}^{(1)} & \cdots & b_{1m}^{(m-1)} & 0 \\
\vdots &   & \vdots & \vdots &    &    &    & \vdots & \vdots \\
b_{11}^{(1)} & \cdots & b_{11}^{(m-1)} & 0 & \cdots & b_{1m}^{(1)} & \cdots & b_{1m}^{(m-1)} & 0 \\
b_{21}^{(1)} & \cdots & b_{21}^{(m-1)} & 0 & \cdots & b_{2m}^{(1)} & \cdots & b_{2m}^{(m-1)} & 0 \\
\vdots &   & \vdots & \vdots &  & \vdots &  & \vdots & \vdots \\
b_{21}^{(1)} & \cdots & b_{21}^{(m-1)} & 0 & \cdots & b_{2m}^{(1)} & \cdots & b_{2m}^{(m-1)} & 0 \\
\tbf{\vdots} &   & \tbf{\vdots} & \tbf{\vdots} &    &  \tbf{\vdots}  &    & \tbf{\vdots} & \tbf{\vdots} \\
b_{m1}^{(1)} & \cdots & b_{m1}^{(m-1)} & 0 & \cdots & b_{mm}^{(1)} & \cdots & b_{mm}^{(m-1)} & 0 \\
\vdots &   & \vdots & \vdots &  & \vdots &  & \vdots & \vdots \\
b_{m1}^{(1)} & \cdots & b_{m1}^{(m-1)} & 0 & \cdots & b_{mm}^{(1)} & \cdots & b_{mm}^{(m-1)} & 0 \\
\end{pmatrix}.
\end{equation}

We are now ready to prove the following theorem.

\begin{theorem} \label{thm:LOT} Let the entries of the matrix $\pac B(t,\xi)$ fulfill the conditions
 \begin{equation} \label{eq:CondLot}
		\sum_{k=1}^m|b_{kj}^{(l)}(t,\xi)|^2\prec \sum_{i=1}^m |\sigma_{m-l}^{(m-1)}(\pi_i \lambda)|^2
	\end{equation} for any $l=1,\dots, m-1$ and $j=1,\dots,m$. Then we have \begin{equation*}
		|\pac W \pac B V|\prec |\pac W V|
	\end{equation*} for all $V \in \C^{m^2}$. More precisely, we define \begin{equation} \begin{aligned}[l]
		& \Sigma_h^{\delta_h} := \Big\{ V \in \C^{m^2} : \quad \sum_{j=h+1}^{m-1} \sum_{i=1}^{m} |\sigma_{m-j}^{(m-1)}(\pi_i \lambda)|^2 \sum_{l=0}^{m-1} |V_{j+lm}|^2 \\ 
		& \qquad \qquad \qquad \qquad \leq \delta_h\sum_{i=1}^m |\sigma_{m-h}^{(m-1)} (\pi_i \lambda)|^2 \sum_{l=0}^{m-1} |V_{h+lm}|^2 \Big\}
	\end{aligned}
	\end{equation}  for $h = 1,\dots, m-2$. There exist suitable $\delta_h$, $h=1,\dots,m-2$ such that \begin{equation*}\begin{aligned}
	|\pac W \pac B V|^2 &\prec \sum_{i=1}^m |\sigma_{m-1}^{(m-1)}(\pi_i \lambda)|^2 \sum\limits_{l=0}^{m-1} |V_{1+lm}|^2 \\
	 |\pac W V|^2 &\succ \sum_{i=1}^m |\sigma_{m-1}^{(m-1)}(\pi_i \lambda)|^2 \sum\limits_{l=0}^{m-1} |V_{1+lm}|^2
	\end{aligned}\end{equation*} on $\Sigma_1^{\delta_1}$ and \begin{equation*}\begin{aligned}
		|\pac W \pac B V|^2 &\prec \sum_{i=1}^m |\sigma_{m-h}^{(m-1)}(\pi_i \lambda)|^2 \sum\limits_{l=0}^{m-1} |V_{h+lm}|^2 \\
		 |\pac W V|^2 &\succ \sum_{i=1}^m |\sigma_{m-h}^{(m-1)}(\pi_i \lambda)|^2 \sum\limits_{l=0}^{m-1} |V_{h+lm}|^2
		\end{aligned}\end{equation*} on $\big(\Sigma_1^{\delta_1}\big)^{\rm{c}} \cap\big(\Sigma_2^{\delta_2}\big)^{\rm{c}} \cap \cdots \cap \big(\Sigma_{h-1}^{\delta_{h-1}}\big)^{\rm{c}} \cap\Sigma_h^{\delta_h}$ for $2 \leq h \leq m-2$. Finally, \begin{equation*}\begin{aligned}
		|\pac W \pac B V|^2 &\prec \sum_{i=1}^m |\sigma_{1}^{(m-1)}(\pi_i \lambda)|^2 \sum\limits_{l=0}^{m-1} |V_{m-1+lm}|^2 \\
		|\pac W V|^2 &\succ \sum_{i=1}^m |\sigma_{1}^{(m-1)}(\pi_i \lambda)|^2 \sum\limits_{l=0}^{m-1} |V_{m-1+lm}|^2
		\end{aligned}\end{equation*} on $\big(\Sigma_1^{\delta_1}\big)^{\rm{c}} \cap\big(\Sigma_2^{\delta_2}\big)^{\rm{c}} \cap \cdots \cap \big(\Sigma_{m-2}^{\delta_{m-2}}\big)^{\rm{c}}$.
\end{theorem}

Note that if $m=2$ no zone argument is needed to prove the theorem above (see Subsection \ref{lot_2_zone}) and when $m=3$ just one zone is needed (see Subsection \ref{lot_3_zone}). The proof of Theorem \ref{thm:LOT} has the same structure as the proof of Theorem 5 in \cite{GR} and requires some auxiliary lemmas.

\begin{lemma} \label{lemma1}
For all $i$ and $j$ with $1\le i,j\le m$ and $k=1,...,m-1,$ one has \begin{equation} \label{formula_diff} \begin{aligned}
& \sigma_{m-k}^{(m-1)}(\pi_i\lambda)-\sigma_{m-k}^{(m-1)}(\pi_j\lambda)\\
& \qquad = (-1)^{m-k}(\lambda_j-\lambda_i)
\sum_{\substack{i_h\neq i,\, i_h\neq j\\ 1\le i_1<i_2<\cdots<i_{m-k-1}\le m}} \lambda_{i_1}\lambda_{i_2}\cdots\lambda_{i_{m-k-1}}
\end{aligned} \end{equation}
\end{lemma}

\begin{proof} The proof can be found in \cite[Lemma 3]{GR}. $\blacksquare$ \end{proof}

\begin{lemma}
\label{lemma2}
For all $k=1,...,m$, we have \begin{equation} \label{formula_k}
\sum_{l=0}^{m-1} \sum_{i=1}^m\biggl|\sum_{j=k}^m \sigma_{m-j}^{(m-1)}(\pi_i\lambda) V_{j+lm} \biggr|^2
\succ
\sum_{i=1}^m|\sigma_{m-k}^{(m-1)}(\pi_i\lambda)|^2 \sum_{l=0}^{m-1}|V_{k+lm}|^2.
\end{equation}
\end{lemma}

\begin{proof} The proof of this lemma follows by induction by applying Lemma \ref{lemma1} and can also be obtained by repeated application of Lemma 4 in \cite{GR} to the respective groups of $V_i$. $\blacksquare$ \end{proof}

\begin{proof} \emph{of Theorem \ref{thm:LOT}.}

By the definition of $\pac B$, we have that $|\pac W \pac B V|^2\prec |\pac WV|^2$ is equivalent to \begin{equation} \label{eq_est}
\begin{aligned}
\sum_{i=1}^m \biggl| \sum_{j=1}^{m-1}\sum_{l=1}^m b_{il}^{(j)} V_{j+(l-1)m} \biggr|^2\prec   \sum_{l=0}^{m-1} \sum_{i=1}^m \biggl|\sum_{j=1}^m \sigma^{(m-1)}_{m-j}(\pi_i\lambda)V_{j+lm} \biggr|^2.
\end{aligned}
\end{equation}

Making use of the Levi-type conditions \eqref{eq:CondLot}, we obtain \begin{eqnarray}
\nonumber \sum_{i=1}^m \biggl| \sum_{j=1}^{m-1}\sum_{l=1}^m b_{il}^{(j)} V_{j+(l-1)m} \biggr|^2 &\prec& \sum_{l=1}^m \sum_{j=1}^{m-1} \biggl( \sum_{i=1}^m |b_{il}^{(j)}|^2 \biggr) |V_{j+(l-1)m}|^2 \\
\label{eq:aux5} &\prec&  \sum_{j=1}^{m-1} \sum_{i=1}^m |\sigma_{m-j}^{(m-1)}(\pi_i \lambda)|^2 \sum_{l=0}^{m-1} |V_{j+lm}|^2.
\end{eqnarray}

On $\Sigma_1^{\delta_1}$, we further obtain the estimate \begin{equation*}
\sum_{i=1}^m \biggl| \sum_{j=1}^{m-1}\sum_{l=1}^m b_{il}^{(j)} V_{j+(l-1)m} \biggr|^2\prec (1+\delta_1) \sum_{i=1}^m |\sigma_{m-1}^{(m-1)}(\pi_i \lambda)|^2 \sum\limits_{l=0}^{m-1} |V_{1+lm}|^2.
\end{equation*}

Lemma \ref{lemma2} gives, setting $k=1$ in \eqref{formula_k} that \begin{equation*}
\sum_{l=0}^{m-1} \sum_{i=1}^m \biggl|\sum_{j=1}^m \sigma^{(m-1)}_{m-j}(\pi_i\lambda)V_{j+lm} \biggr|^2 \succ
\sum_{i=1}^m|\sigma_{m-1}^{(m-1)}(\pi_i\lambda)|^2 \sum_{l=0}^{m-1}|V_{1+lm}|^2.
\end{equation*}

This proves inequality \eqref{eq_est} in $\Sigma_1^{\delta_1}$. Now, we assume that $V \in (\Sigma_1^{\delta_1})^c \cap (\Sigma_2^{\delta_2})^c \cap \dots \cap (\Sigma_{h-1}^{\delta_{h-1}})^c \cap \Sigma_h^{\delta_h}$ for $2 \leq h \leq m-2$. From the definition of the zones for $1\le k\le h-1$ and $\delta_k\ge 1$, we obtain \begin{equation*}
\begin{aligned}
&\sum_{i=1}^m |\sigma^{(m-1)}_{m-(h-1)}(\pi_i\lambda)|^2 \sum_{l=0}^{m-1} |V_{h-1+lm}|^2 \\
& \qquad <\frac{1}{\delta_{h-1}}\biggl( \sum_{j=h+1}^{m-1}\sum_{i=1}^m|\sigma^{(m-1)}_{m-j}(\pi_i\lambda)|^2 \sum_{l=0}^{m-1}|V_{j+lm}|^2 \\
& \qquad \qquad +\sum_{i=1}^m|\sigma^{(m-1)}_{m-h}(\pi_i\lambda)|^2 \sum_{l=0}^{m-1}|V_{h+lm}|^2 \biggl) \\
& \qquad \le \frac{1}{\delta_{h-1}}(1+\delta_h)\sum_{i=1}^m|\sigma^{(m-1)}_{m-h}(\pi_i\lambda)|^2 \sum_{l=0}^{m-1} |V_{h+lm}|^2,
\end{aligned}
\end{equation*} as well as \begin{equation*}
\begin{aligned}[l]
& \sum_{i=1}^m |\sigma^{(m-1)}_{m-(h-2)}(\pi_i\lambda)|^2 \sum_{l=0}^{m-1} |V_{h-2+lm}|^2\\
& \qquad <\frac{1}{\delta_{h-2}}\biggl(\sum_{j=h+1}^{m-1}\sum_{i=1}^m|\sigma^{(m-1)}_{m-j}(\pi_i\lambda)|^2 \sum_{l=0}^{m-1} |V_{j+lm}|^2\\
& \qquad \qquad +\sum_{i=1}^m|\sigma^{(m-1)}_{m-h}(\pi_i\lambda)|^2 \sum_{l=0}^{m-1} |V_{h+lm}|^2 \\
& \qquad \qquad + \sum_{i=1}^m|\sigma^{(m-1)}_{m-(h-1)}(\pi_i\lambda)|^2 \sum_{l=0}^{m-1} |V_{h-1+lm}|^2 \biggl)\\
& \qquad \le\frac{1}{\delta_{h-2}}\big(1+\delta_h+\frac{1}{\delta_{h-1}}(1+\delta_h)\big) \sum_{i=1}^m|\sigma^{(m-1)}_{m-h}(\pi_i\lambda)|^2 \sum_{l=0}^{m-1} |V_{h+lm}|^2\\
& \qquad \le (1+\delta_h)\big(\frac{1}{\delta_{h-1}}+\frac{1}{\delta_{h-2}})\sum_{i=1}^m|\sigma^{(m-1)}_{m-h}(\pi_i\lambda)|^2 \sum_{l=0}^{m-1} |V_{h+lm}|^2.
\end{aligned}
\end{equation*}

Continuing these estimates recursively, we obtain that \begin{equation} \label{bound}
\begin{aligned}
& \sum_{i=1}^m |\sigma^{(m-1)}_{m-j}(\pi_i\lambda)|^2 \sum_{l=0}^{m-1} |V_{j+lm}|^2\\
& \qquad \qquad\prec (1+\delta_h)\sum_{k=1}^{h-1}\frac{1}{\delta_k}\sum_{i=1}^m|\sigma^{(m-1)}_{m-h}(\pi_i\lambda)|^2 \sum_{l=0}^{m-1} |V_{h+lm}|^2
\end{aligned}
\end{equation} for all $j$ with $1\le j \le h-1$ is valid on the zone $\big(\Sigma_1^{\delta_1}\big)^{\rm{c}} \cap\big(\Sigma_2^{\delta_2}\big)^{\rm{c}} \cap \cdots \cap \big(\Sigma_{h-1}^{\delta_{h-1}}\big)^{\rm{c}} \cap\Sigma_h^{\delta_h}$.

From \eqref{eq:aux5}, the estimate \eqref{bound} and the definition of the zone $\big(\Sigma_1^{\delta_1}\big)^{\rm{c}} \cap\big(\Sigma_2^{\delta_2}\big)^{\rm{c}} \cap \cdots \cap \big(\Sigma_{h-1}^{\delta_{h-1}}\big)^{\rm{c}} \cap\Sigma_h^{\delta_h}$ we get the following estimate of the left-hand side of \eqref{eq_est}:
\begin{equation*}
\begin{aligned}
&\sum_{i=1}^m \biggl| \sum_{j=1}^{m-1}\sum_{l=1}^m b_{il}^{(j)} V_{j+(l-1)m} \biggr|^2\prec  \sum_{j=1}^{m-1} \sum_{i=1}^m |\sigma_{m-j}^{(m-1)}(\pi_i \lambda)|^2 \sum_{l=0}^{m-1} |V_{j+lm}|^2 \\
& \quad\prec \sum_{j=h+1}^{m-1} \sum_{i=1}^m |\sigma_{m-j}^{(m-1)}(\pi_i\lambda)|^2 \sum_{l=0}^{m-1} |V_{h+lm}|^2 \\
& \qquad +\sum_{i=1}^m |\sigma_{m-h}^{(m-1)}(\pi_i\lambda)|^2 \sum_{l=0}^{m-1} |V_{h+lm}|^2
+\sum_{j=1}^{h-1}\sum_{i=1}^m |\sigma_{m-j}^{(m-1)}(\pi_i\lambda)|^2 \sum_{l=0}^{m-1} |V_{j+lm}|^2 \\
& \quad\prec \sum_{i=1}^m |\sigma_{m-h}^{(m-1)}(\pi_i\lambda)|^2 \sum_{l=0}^{m-1} |V_{h+lm}|^2.
\end{aligned}
\end{equation*}

Now, we have to estimate the right-hand side of \eqref{eq_est} on $\big(\Sigma_1^{\delta_1}\big)^{\rm{c}} \cap\big(\Sigma_2^{\delta_2}\big)^{\rm{c}} \cap \cdots \cap \big(\Sigma_{h-1}^{\delta_{h-1}}\big)^{\rm{c}} \cap\Sigma_h^{\delta_h}$. We make use of Lemma \ref{lemma2} and of the bound \eqref{bound}. We obtain \begin{equation*}
\begin{aligned}
& \sum_{l=0}^{m-1} \sum_{i=1}^m\biggl|\sum_{j=1}^m  \sigma^{(m-1)}_{m-j}(\pi_i\lambda)V_{j+lm}\biggr|^2\\
& \quad \succ \gamma_1 \sum_{l=0}^{m-1} \sum_{i=1}^m \biggl|\sum_{j=h}^{m}\sigma^{(m-1)}_{m-j}(\pi_i\lambda)  V_{j+lm}
\biggr|^2 -\gamma_2 \sum_{l=0}^{m-1} \sum_{i=1}^m \biggl|\sum_{j=1}^{h-1}\sigma^{(m-1)}_{m-j}(\pi_i\lambda)  V_{j+lm}
\biggr|^2\\
& \quad \succ \gamma_1 \sum_{l=0}^{m-1} \sum_{i=1}^m \biggl|\sum_{j=h}^{m}\sigma^{(m-1)}_{m-j}(\pi_i\lambda)  V_{j+lm}
\biggr|^2 \\
& \qquad \quad -\gamma_2 \sum_{i=1}^m \sum_{j=1}^{h-1} |\sigma^{(m-1)}_{m-j}(\pi_i\lambda)|^2 \sum_{l=0}^{m-1}  |V_{j+lm}|^2 \\
& \quad \succ \gamma_1  \sum_{i=1}^m|\sigma_{m-h}^{(m-1)}(\pi_i\lambda)|^2 \sum_{l=0}^{m-1} |V_{h+lm}|^2 \\
& \qquad \quad -\gamma_2
(1+\delta_h) \sum_{k=1}^{h-1}\frac{1}{\delta_k}\sum_{i=1}^m|\sigma^{(m-1)}_{m-h}(\pi_i\lambda)|^2\sum_{l=0}^{m-1}|V_{h+lm}|^2\\
& \quad =\biggl(\gamma_1-\gamma_2(1+\delta_h)  \sum_{k=1}^{h-1}\frac{1}{\delta_k}\biggr)
\sum_{i=1}^m|\sigma^{(m-1)}_{m-h}(\pi_i\lambda)|^2 \sum_{l=0}^{m-1}|V_{h+lm}|^2,
\end{aligned}
\end{equation*} where the second inequality follows from \begin{equation*}
	\sum_{l=0}^{m-1} \sum_{i=1}^m \biggl|\sum_{j=1}^{h-1}\sigma^{(m-1)}_{m-j}(\pi_i\lambda)  V_{j+lm}
	\biggr|^2 \leq (h-1) \sum_{i=1}^m \sum_{j=1}^{h-1} |\sigma^{(m-1)}_{m-j}(\pi_i\lambda)|^2 \sum_{l=0}^{m-1}  |V_{j+lm}|^2
\end{equation*} which follows from $|z_1 + \cdots + z_k| \leq k \sum_{i=1}^k |z_i|^2$. This yields estimate \eqref{eq_est} on the zone $\big(\Sigma_1^{\delta_1}\big)^{\rm{c}} \cap\big(\Sigma_2^{\delta_2}\big)^{\rm{c}} \cap \cdots \cap \big(\Sigma_{h-1}^{\delta_{h-1}}\big)^{\rm{c}} \cap\Sigma_h^{\delta_h}$ for any $\delta_h > 0$ provided that $\delta_1$, $\dots$, $\delta_{h-1}$ are chosen large enough.

The last step is assuming that $V \in \big(\Sigma_1^{\delta_1}\big)^{\rm{c}} \cap \big(\Sigma_2^{\delta_2}\big)^{\rm{c}} \cap \cdots \cap
\big(\Sigma_{m-2}^{\delta_{m-2}}\big)^{\rm{c}}$. Thus, from the definition of the $\Sigma^{\delta_h}$, we have \begin{equation} \label{eq:AuxProofLOT2} \begin{aligned}
&\sum_{j=h+1}^{m-1}\sum_{i=1}^{m-1} |\sigma^{(m-1)}_{m-j}(\pi_i\lambda)|^2 \sum_{l=0}^{m-1} |V_{j+lm}|^2 \\
& \qquad > \delta_h \sum_{i=1}^m|\sigma^{(m-1)}_{m-h}(\pi_i\lambda)|^2 \sum_{l=0}^{m-1} |V_{h+lm}|^2
\end{aligned} \end{equation} for $1 \leq h \leq m-2$. More precisely from the previous estimate we obtain $m-2$ inequalities starting with 
\begin{equation} \label{eq:AuxProofLOT1} \begin{aligned}
& \sum_{i=1}^m|\sigma^{(m-1)}_{m-1}(\pi_i\lambda)|^2 \sum_{l=0}^{m-1} |V_{1+lm}|^2 \\
& \qquad < \frac{1}{\delta_1} \sum_{j=2}^{m-1}\sum_{i=1}^{m-1} |\sigma^{(m-1)}_{m-j}(\pi_i\lambda)|^2 \sum_{l=0}^{m-1} |V_{j+lm}|^2,
\end{aligned} \end{equation} 
(where we put $h=1$ in \eqref{eq:AuxProofLOT2}) and ending with \begin{equation*} \begin{aligned}
& \sum_{i=1}^m|\sigma^{(m-1)}_{2}(\pi_i\lambda)|^2 \sum_{l=0}^{m-1} |V_{(m-2)+lm}|^2 \\
& \qquad < \frac{1}{\delta_{m-2}} \sum_{i=1}^{m-1} |\sigma^{(m-1)}_{1}(\pi_i\lambda)|^2 \sum_{l=0}^{m-1} |V_{(m-1)+lm}|^2,
\end{aligned} \end{equation*} (where $h=m-2$ in \eqref{eq:AuxProofLOT2}). 
Using now the second of the inequalities, i.e. $h=2$ in \eqref{eq:AuxProofLOT2}, on the right hand side of \eqref{eq:AuxProofLOT1}, we get\begin{eqnarray*}
	&& \frac{1}{\delta_1} \sum_{j=3}^{m-1}\sum_{i=1}^{m-1} |\sigma^{(m-1)}_{m-j}(\pi_i\lambda)|^2 \sum_{l=0}^{m-1} |V_{j+lm}|^2 + \frac{1}{\delta_1} \sum_{i=1}^m |\sigma^{(m-1)}_{m-2}(\pi_i\lambda)|^2 \sum_{l=0}^{m-1} |V_{2+lm}|^2 \\
	&& \qquad \leq \left( \frac{1}{\delta_1} + \frac{1}{\delta_1} \frac{1}{\delta_2} \right) \sum_{j=3}^{m-1}\sum_{i=1}^{m-1} |\sigma^{(m-1)}_{m-j}(\pi_i\lambda)|^2 \sum_{l=0}^{m-1} |V_{j+lm}|^2.
\end{eqnarray*} Then using the remaining estimates for $h=3$ to $h=m-2$ recursively, we finally arrive at \begin{equation} \label{eq:AuxEst6} \begin{aligned}
&\sum_{i=1}^m |\sigma_{m-j}^{(m-1)}(\pi_i \lambda)|^2 \sum_{l=0}^{m-1} |V_{j+lm}|^2 \\
& \qquad \leq \sum_{h=1}^{m-2} \frac{1}{\delta_h} \sum_{i=1}^{m} |\sigma_{1}^{(m-1)}(\pi_i \lambda)|^2 \sum_{l=0}^{m-1} |V_{m-1+lm}|^2
\end{aligned}
\end{equation} for any $1 \leq j \leq m-2$, $\delta_h\geq 1$. From \eqref{eq:AuxEst6} and the Levi-type conditions we deduce that \begin{equation*}  \begin{aligned}
\sum_{i=1}^m \biggl| \sum_{j=1}^{m-1}\sum_{l=1}^m b_{il}^{(j)} V_{j+(l-1)m} \biggr|^2\prec \sum_{i=1}^m |\sigma_1^{(m-1)}(\pi_i \lambda)|^2 \sum_{l=0}^{m-1} |V_{m-1+lm}|^2
\end{aligned} \end{equation*} in $\big(\Sigma_1^{\delta_1}\big)^{\rm{c}} \cap \big(\Sigma_2^{\delta_2}\big)^{\rm{c}} \cap \cdots \cap
\big(\Sigma_{m-2}^{\delta_{m-2}}\big)^{\rm{c}}$.

Using Lemma \ref{lemma2}, we get \begin{equation*} \begin{aligned}
& \sum_{l=0}^{m-1} \sum_{i=1}^m\biggl|\sum_{j=1}^m  \sigma^{(m-1)}_{m-j}(\pi_i\lambda)V_{j+lm}\biggr|^2 \\
& \qquad = \sum_{l=0}^{m-1} \sum_{i=1}^m\biggl| \sum_{j=1}^{m-2}  \sigma^{(m-1)}_{m-j}(\pi_i\lambda)V_{j+lm} + \sum_{j=m-1}^{m}  \sigma^{(m-1)}_{m-j}(\pi_i\lambda)V_{j+lm}\biggr|^2 \\
& \qquad \succ \gamma_1 \sum_{l=0}^{m-1} \sum_{i=1}^m\biggl|\sum_{j=m-1}^m \sigma^{(m-1)}_{m-j}(\pi_i\lambda)V_{j+lm}\biggr|^2 - \gamma_2 \sum_{l=0}^{m-1} \sum_{i=1}^m\biggl| \sum_{j=1}^{m-2}  \sigma^{(m-1)}_{m-j}(\pi_i\lambda)V_{j+lm}\biggr|^2 \\
& \qquad \succ \gamma_1 \sum_{i=1}^m |\sigma_1^{(m-1)}(\pi_i \lambda)|^2 \sum_{l=0}^{m-1} |V_{m-1+lm}|^2 \\
& \qquad \qquad -\gamma_2 \sum_{l=0}^{m-1} \sum_{i=1}^m\biggl| \sum_{j=1}^{m-2}  \sigma^{(m-1)}_{m-j}(\pi_i\lambda)V_{j+lm}\biggr|^2.
\end{aligned} \end{equation*} The second term on the right-hand side of the last inequality can be estimated with \eqref{eq:AuxEst6} and we obtain \begin{equation*}
\sum_{l=0}^{m-1} \sum_{i=1}^m\biggl|\sum_{j=1}^m  \sigma^{(m-1)}_{m-j}(\pi_i\lambda)V_{j+lm}\biggr|^2 \succ \sum_{i=1}^m |\sigma_1^{(m-1)}(\pi_i \lambda)|^2 \sum_{l=0}^{m-1} |V_{m-1+lm}|^2
\end{equation*} provided that the $\delta_h$, $1\leq h\leq m-2$ are chosen large enough. Thus \eqref{eq_est} holds on the zone $\big(\Sigma_1^{\delta_1}\big)^{\rm{c}} \cap \big(\Sigma_2^{\delta_2}\big)^{\rm{c}} \cap \cdots \cap
\big(\Sigma_{m-2}^{\delta_{m-2}}\big)^{\rm{c}}$ and the proof of Theorem \ref{thm:LOT} is complete. 
$\blacksquare$
\end{proof}

\section{Well-posedness results} \label{sec_final}

In this section we prove our main result: the well-posedness of the Cauchy problem \eqref{CauchyP}. We formulate the following theorem by adopting the language and the notations of the previous sections concerning the lower order terms. A different formulation will be given in Theorem \ref{theo_GJ_2}. Note that Theorem \ref{theo_GJ} and Theorem \ref{theo_GJ_2} correspond to Theorem \ref{theo_GJ_intro} and Theorem \ref{theo_GJ_2_intro}, respectively.
\begin{theorem}
\label{theo_GJ}
Let $A(t,D_x)$, $t\in[0,T]$, $x\in\R^n$, be an $m\times m$ matrix of first order differential operators with $C^\infty$-coefficients. Let $A(t,\xi)$ have real eigenvalues satisfying condition \eqref{eq:CondKinSp}. Let 
\begin{equation*}
\left\{ \begin{array}{ll}
& D_t u - A(t,D_x)u = 0,~(t,x) \in [0,T] \times \R^n\\
& \left. u \right|_{t=0} = u_0,~x \in \R^n
\end{array} \right.
\end{equation*}
be the Cauchy problem \eqref{CauchyP}. Assume that the Cauchy problem \eqref{eq:TrafoSys},
\[
	\left\{ \begin{array}{l}
	D_t U = \pazocal A(t,D_x) U + \pazocal B(t,D_x)U, \\
	\left. U \right|_{t=0} = U_0=(U_{0,1}, \cdots, U_{0,m})^T,
	\end{array} \right.
\]
obtained from \eqref{CauchyP} by block Sylvester reduction as in Section \ref{sec:BlockDiagonal} has the lower order terms matrix $ \pazocal B(t,D_x)$  fulfilling the Levi-type conditions \eqref{eq:CondLot}. Hence, for all $s\ge 1$ and for all $u_0\in \gamma^s(\R^n)^m$ there exists a unique solution $u\in C^1([0,T], \gamma^s(\R^n))^m$ of the Cauchy problem \eqref{CauchyP}.
\end{theorem}
\begin{proof}
We assume $s>1$ since the case $s=1$ is known thanks to see
\cite{Jannelli:analytic-CPDE-1984} and \cite{Kajitani:analytic-CPDE-1986}. By the finite propagation speed for hyperbolic equations it is not restrictive to take compactly supported initial data and, therefore, to have the solution $u$ compactly supported in $x$. Note that if $u_0\in \gamma_c^s(\R^n)^m$ then by deriving the system in \eqref{CauchyP} with respect to $t$ we immediately have that $D_t^j u(0,x)\in\gamma_c^s(\R^n)^m$ for $j=1,\dots,m-1$. It follows that if $u$ solves \eqref{CauchyP} then $U$ defined in \eqref{trans_T} solves the Cauchy problem \eqref{eq:TrafoSys} with initial data $U_0\in\gamma_c^s(\R^n)^{m^2}$. We now prove that $U\in C^1([0,T],\gamma^s(\R^n))^{m^2}$. This will allow us to conclude that $u\in C^1([0,T], \gamma^s(\R^n))^m$. We recall that the Cauchy problem \eqref{eq:TrafoSys} is given by the system
\[
D_t U = \pazocal A(t,D_x) U + \pazocal B(t,D_x)U,
\] 
where $\pazocal A(t,\xi)$ is a block Sylvester matrix with $m$ identical blocks having the same eigenvalues of $A(t,\xi)$. We make use of the energy $E_\eps$ defined via the quasi-symmetriser in Section \ref{sec_energy}. Combining the energy estimate \eqref{eq:EE} with the estimates of the first, second and third term in Subsection \ref{first_term}, \ref{second_term} and \ref{third_term}, respectively,  we get
\begin{equation}\label{est-rev}
\partial_t E_\eps(t,\xi)\le (K_\eps(t,\xi)+C_2\eps\lara{\xi}+C_3)E_\eps(t,\xi),
\end{equation}
where $K_\eps(t,\xi)$ is defined in Subsection \ref{first_term}, the bound from above 
\[
\int_0^T K_\eps(t,\xi)\, dt\le C_1\eps^{-2(m-1)/k},
\]
holds for all $k\ge 1$ and $C_1, C_2, C_3$ are positive constants. Note that in the estimate \eqref{est-rev} we have used both the condition \eqref{eq:CondKinSp} on the eigenvalues and the Levi-type conditions \eqref{eq:CondLot}. Thanks to the reduction to block Sylvester form that we have applied to obtain the Cauchy problem \eqref{eq:TrafoSys}, we deal here with the same kind of energy employed in \cite{GR} for the scalar weakly hyperbolic equations of order $m$. The proof therefore continues as the proof of Theorem 6 in \cite{GR} with the only difference that $k$ can be taken arbitrary. This is due to the fact that the coefficients of the matrix $A(t,\xi)$ are $C^\infty$ with respect to $t$. It follows, by working on the Fourier transform level, that $U\in C^1([0,T],\gamma^s(\R^n))^{m^2}$ and therefore $u\in C^1([0,T], \gamma^s(\R^n))^m$.
$\blacksquare$
\end{proof}
We now formulate Theorem \ref{theo_GJ} with an additional condition on the matrix $A(t,\xi)$ which implies the Levi-type conditions \eqref{eq:CondLot}.  \begin{theorem}
\label{theo_GJ_2}
Let $A(t,D_x)$, $t\in[0,T]$, $x\in\R^n$, be an $m\times m$ matrix of first order differential operators with $C^\infty$-coefficients. Let $A$ have real eigenvalues satisfying condition \eqref{eq:CondKinSp} and let $Q=(q_{ij})$ be the symmetriser of $A_0=\lara{\xi}^{-1}A$. Assume that 
\begin{equation}
\label{cond_LOT_2}
\max_{k=1,\dots,m-1}\Vert D_t^k A_0(t,\xi)\Vert^2 \prec q_{j,j}(t,\xi)
\end{equation}
for all $(t,\xi) \in [0,T] \times \R^n$ and $j=1,\dots,m-1$. Hence, for all $s\ge 1$ and for all $u_0\in \gamma^s(\R^n)^m$ there exists a unique solution $u\in C^1([0,T], \gamma^s(\R^n))^m$ of the Cauchy problem \eqref{CauchyP}.
 
\end{theorem}
\begin{proof}
From Proposition \ref{prop_B_lot} and Corollary \ref{cor:LOT} we have that \begin{equation*}
\sum_{k=1}^m|b_{kj}^{(l)}(t,\xi)|^2 \prec \max_{k=1,\dots,m-1}\Vert D_t^k A_0(t,\xi)\Vert^2
\end{equation*} for all $(t,\xi) \in [0,T] \times \R^n$ and $l=1,\dots, m-1$ and $j=1,\dots,m$. It follows that \eqref{cond_LOT_2} implies the Levi-type conditions \eqref{eq:CondLot} and therefore Theorem \ref{theo_GJ_2} follows from Theorem \ref{theo_GJ}.
$\blacksquare$
\end{proof}

It is clear that the hypothesis \eqref{cond_LOT_2} on the matrix $A_0=A\lara{\xi}^{-1}$ is in general stronger than the Levi-type conditions  \eqref{eq:CondLot}. However, in some cases  \eqref{eq:CondLot} and \eqref{cond_LOT_2} coincide as illustrated by the following examples.

\begin{example} In the special case $D_t^2 u - a(t)D_x^2 u = 0$ with $a(t) \geq 0$ and appropriate Cauchy data, the Levy-type condition is automatically satisfied for $a \in C^2[0,T]$. Indeed, with $a_{11}=0$, $a_{12}=1$, $a_{21}=a(t)$, and $a_{22} = 0$, condition \eqref{cond:LevytypeEx} becomes $|D_t a(t)| \leq C a(t)$ which is satisfied by Glaeser's inequality \cite{Glaeser}. \end{example}

\begin{example}
When $m=2$, the Levi-type conditions \eqref{eq:CondLot} imply \eqref{cond_LOT_2} (and therefore coincide with it). Indeed, as observed in Subsection \ref{lot_2_zone}, the Levi-type conditions are formulated as 
\begin{eqnarray*}
	(|D_t a_{11}(t)|^2 + |D_t a_{21}(t)|^2)\lara{\xi}^{-2}&\prec & \lambda_1^2(t,\xi) + \lambda_2^2(t,\xi), \\
	(|D_t a_{12}(t)|^2 + |D_t a_{22}(t)|^2)\lara{\xi}^{-2}&\prec& \lambda_1^2(t,\xi) + \lambda_2^2(t,\xi).
\end{eqnarray*} This implies 
\[
\Vert D_t A_0\Vert^2 \prec q_{1,1}
\]
which is condition \eqref{cond_LOT_2}.

\end{example}
\begin{example}
Let us now take a $3\times 3$ matrix $A$ with trace zero. For simplicity let us assume that $n=1$ and that the eigenvalues of the corresponding $A_0$ are $\lambda_1(t,\xi)=-\sqrt{a(t)}\xi\lara{\xi}^{-1}$, $\lambda_2(t,\xi)=0$ and $\lambda_3(t,\xi)=\sqrt{a(t)}\xi\lara{\xi}^{-1}$ with $a(t)\ge 0$ for $t\in[0,T]$. It follows that the hypothesis \eqref{eq:CondKinSp} on the eigenvalues is satisfied. By direct computations we get
\[
\begin{split}
q_{1,1}&=\lambda_1^2\lambda_2^2+\lambda_1^2\lambda_3^2+\lambda_2^2\lambda_3^2=a(t)\xi^2\lara{\xi}^{-2},\\
q_{2,2}&=(\lambda_1+\lambda_2)^2+(\lambda_1+\lambda_3)^2+(\lambda_2+\lambda_3)^2=2a(t)\xi^2\lara{\xi}^{-2}.
\end{split}
\]
It follows that both $q_{1,1}$ and $q_{2,2}$ are comparable to $a$ and therefore combining \eqref{eq:condm3} with \eqref{form_b_3} we conclude that
\[
\begin{split}
	|b^{(1)}_{kj}|^2 &= |D_t^2 a_{kj} + 2D_t a_{kj}|^2\prec a(t),\\
 |b^{(2)}_{kj}|^2& = |a_{k1}D_t a_{1j} + a_{k2}D_t a_{2j} + a_{k3}D_t a_{3j}|^2\prec a(t),
 \end{split}
\]
for $k=1,2,3$ and $j=1,2$.
We can easily see on the matrix
\[
A(t,\xi)=\begin{pmatrix}
0 & a(t) & 0\\
1 & 0 & 0\\
0 &  1 & 0
\end{pmatrix}\xi
\]
that the conditions above on the entries of $A$ entail 
\[
|D_t^k a(t)|^2\prec a(t)
\]
for all $t\in[0,T]$ and $k=1,2$, i.e. condition  \eqref{cond_LOT_2} . 
\end{example}

We now assume that the coefficients of the matrix $A(t,\xi)$ are analytic with respect to $t$. We will prove that in this case the Cauchy problem \eqref{CauchyP} with the same Levi-type conditions employed above is $C^\infty$ well-posed.

The proof of the $C^\infty$ well-posedness follows very closely the arguments in \cite{GR}. Thus, we will only give a sketch with the differences and refer the reader to the cited work for more details. We begin by recalling a lemma on analytic functions whose proof can be found in \cite{GR} (see Lemma 5 in \cite{GR}). 

\begin{lemma}  
\label{lem:Analytic} 
Let $f(t,\xi)$ be an analytic function in $t \in [0,T]$,  continuous and homogeneous of order $0$ in $\xi \in \R^n$. Then, \begin{enumerate}[{\quad (}i{)}]
		\item for all $\xi$ there exists a finite partition $(\tau_{h(\xi)})$ of the interval $[0,T]$ such that \begin{equation*}
			0 = \tau_0 < \tau_1 < \dots < \tau_{h(\xi)}< \dots < \tau_{N(\xi)} = T
		\end{equation*} with $\sup_{\xi \neq 0} N(\xi) < +\infty$, such that $f(t,\xi) \neq 0$ in each open interval $(\tau_{h(\xi)}, \tau_{(h+1)(\xi)})$;
		\item there exists a positive constant $C$ such that \begin{equation*}
			|\partial_t f(t,\xi)| \leq C \left( \frac{1}{t-\tau_{h(\xi)}} + \frac{1}{\tau_{(h+1)(\xi)}-t} \right) |f(t,\xi)|
		\end{equation*} for all $t \in (\tau_{h(\xi)}, \tau_{(h+1)(\xi)})$, $\xi \in \R^n \setminus \{0\}$ and $0 \leq h(\xi) \leq N(\xi)-1$.
	\end{enumerate}
\end{lemma}

\begin{theorem}
\label{theo_GJ_anal}
	If all entries of $A(t,D_x)$ in \eqref{CauchyP} are analytic on $[0,T]$, the eigenvalues satisfy \eqref{eq:CondKinSp} and the entries of the matrix $\pac B(t,\xi)$ in \eqref{eq:TrafoSys} satisfy the Levi conditions \eqref{eq:CondLot} for $\xi$ away from $0$, then the Cauchy problem \eqref{CauchyP} is $C^\infty$ well-posed, i.e., for all $u_0\in C^\infty(\R^n)^m$ there exists a unique solution $u\in C^1([0,T], C^\infty(\R^n))^m$ of the Cauchy problem \eqref{CauchyP}.

\end{theorem}

\begin{proof}
Thanks to the finite propagation speed property it is not restrictive to assume that the initial data have compact support. By Remark \ref{rem:QuasiSym}, the entries of the quasi-symmetriser $\pac Q_\eps^{(m)}(t,\xi)$ are analytic in $t \in [0,T]$ and, using Proposition \ref{prop_qs}, can be written as \begin{equation} \label{eq:AuxProof1}
	q_{\eps,ij}(t,\xi) = q_{0,ij}(t,\xi) + \eps^2 q_{1,ij}(t,\xi) + \dots + \eps^{2(m-1)}q_{m-1,ij}(t,\xi).
\end{equation} We note that $q_{\eps,{(i+hm)(j+hm)}} = q_{\eps,ij}$, $h=0,\dots,m-1$ due to the block-diagonal structure of $\pac Q_\eps^{(m)}(t,\xi)$. Since all functions on the right hand side of \eqref{eq:AuxProof1} are analytic, we can use Lemma \ref{lem:Analytic} on each of them. Note that the partition $(\tau_{h(\xi)})$ in Lemma \ref{lem:Analytic} can be chosen independent from $\eps$.

Now, following \cite{GR,KS}, we use a Kovalevskayan-type energy near the points $\tau_{h(\xi)}$ and a hyperbolic-type energy on the rest of the interval $[0,T]$ (see also \cite{Jannelli}). We start with the interval $[0,\tau_{1}]$ ($\tau_1 = \tau_{1(\xi)}$), setting \begin{equation*}
	E_\eps(t,\xi) = \left\{ \begin{array}{lcl}
	|V(t,\xi)|^2 &\text{for}& t \in [0,\eps] \cup [\tau_1-\eps,\tau_1], \\
	\langle Q_\eps^{(m)}(t,\xi)V(t,\xi) | V(t,\xi)\rangle &\text{for}& t \in [\eps,\tau_1-\eps].
	\end{array} \right.
\end{equation*} The estimate on $[0,\eps] \cup [\tau_1-\eps,\tau_1]$ is standard and the details are left to the reader. We obtain, as in \cite{GR},
\begin{equation} \label{eq:AuxAnal1}
E_\eps(t,\xi) \leq \left\{ \begin{array}{lcl}
e^{2C \eps \lara{\xi}} E_\eps(0,\xi) &\text{for}& t \in [0,\eps] \\
e^{2C\eps\lara{\xi}} E_\eps(\tau_1-\eps,\xi) &\text{for}& t \in [\tau_1-\eps].
\end{array} \right.
\end{equation} 
On $[\eps, \tau_1-\eps]$, we get \begin{equation} \label{eq:AuxAnalytic1}
	\nonumber \partial_t E(t,\xi) \leq \left( \frac{|(\partial_t \pac Q^{(m)}_\eps V,V)|}{(\pac Q_\eps^{(m)} V| V)} + C_2 \eps \lara{\xi} + C_3\right) E_\eps(t,\xi),
\end{equation} where we used \eqref{eq:p3} (see (iii) in Proposition \ref{prop_qs}) and the Levi-type conditions \eqref{eq:CondLot} for $|\xi| \geq R$ to ensure that we have
\[
|((\pac{Q}_0^{(m)} B-B^\ast \pac{Q}_0^{(m)})V | V)| \leq C |\pac W^{(m)} V|^2 = (\pac Q^{(m)}_0 V| V),
 \]
see also \eqref{eq:cB} in Subsection \ref{third_term}. Thanks to Proposition \ref{prop_SM}, the family $\{\pac Q_\eps^{(m)}\}$ is nearly diagonal, when the eigenvalues $\lambda_l$, $l=1,\dots,m$ of $A$ satisfy \eqref{eq:CondKinSp}. Thus, we have $\pac Q_\eps \geq c_0 \diag(\pac Q_\eps^{(m)})$, i.e, \begin{eqnarray*}
 (\pac Q_\eps^{(m)} V | V) \geq c_0 \sum_{h=1}^m q_{\eps,hh} \sum_{l=0}^{m-1} |V_{h+lm}|^2
 = c_0\sum_{h=1}^{m^2} q_{\eps,hh}  |V_h|^2.
\end{eqnarray*} 
Using Proposition \ref{prop_qs} and the Cauchy-Schwarz inequality, we obtain \begin{equation*}
|q_{\eps,ij}||V_i||V_j| \leq \sum_{h=1}^{m^2} q_{\eps,hh}|V_h|^2.
\end{equation*} Together with Lemma \ref{lem:Analytic}, using the last two inequalities, we conclude that \begin{equation*}
\int\limits_{\eps}^{\tau_1-\eps} \frac{|(\partial_t \pac Q^{(m)}_\eps V,V)|}{(\pac Q_\eps^{(m)} V| V)} dt \leq \frac{1}{c_0} \int\limits_\eps^{\tau_1-\eps} \sum_{i,j=1}^{m^2} \frac{|\partial_t q_{ij}(t,\xi)|}{|q_{ij}(t,\xi)|} dt \leq C \log\left( \frac{T}{\eps} \right)
\end{equation*} for a certain positive constant $C$ not depending on $t$ and $\xi$. 
Thanks to the block diagonal form of the quasi-symmetriser, the proof now continues as the proof of Theorem 7 in \cite{GR}. This leads to the inequality
\begin{equation*}
|V(t,\xi)| \leq c \lara{\xi}^{N(\xi)(m-1)} e^{N(\xi)C_T} \lara{\xi}^{N(\xi)C_T},
\end{equation*}
obtained by setting $\eps=\lara{\xi}^{-1}$.
Lemma \ref{lem:Analytic} guarantees that the function $N(\xi)$ is bounded in $\xi$. Therefore, we can conclude that there exists a $\kappa \in \N$, depending only on $n$, $m$, and $T$ as well as a positive constant $C>0$ such that 
\[
|V(t,\xi)| \leq C \lara{\xi}^\kappa |V(0,\xi)|
\]
for all $t \in [0,T]$ and $|\xi| \geq R$. Clearly this estimate implies the $C^\infty$ well-posedness of the Cauchy problem \eqref{CauchyP}. $\blacksquare$
\end{proof}

\begin{remark}
\label{rem_reg}
Since the entries of the matrix $A$ are at least $C^\infty$ with respect to $t$  in both Theorem \ref{theo_GJ} and \ref{theo_GJ_anal}, from the system itself in \eqref{CauchyP} we obtain that the dependence in $t$ of the solution $u$ is actually not only $C^1$ but $C^\infty$.  \end{remark}
\begin{remark}
\label{rem_non_hom}
In this paper we have studied homogeneous systems. Our method, described in the previous sections, can be generalised to non-homogeneous systems with some technical work on the lower order terms. Key point is to investigate the relation of the matrix of the lower order terms in the original system with the matrix $\pac B$ obtained after reduction to block Sylvester form. 
\end{remark}

%

\appendix

\setcounter{section}{11}

\section{Some linear algebra auxiliary results} \label{sec:Appendix}

This appendix contains some general linear algebra results which have been employed throughout the paper. We start with  the following definition.
\begin{definition}[Adjunct/classical adjoint] \label{def:Adjunct}
	Let $A \in \R^{m \times m}$. Then, the adjunct (or classical adjoint) of $A$, denoted $\adj(A)$, is defined as the matrix consisting of the elements \begin{equation*}
		\adj(A)_{ij} = (-1)^{i+j} \det(A_{\hat{j}\hat{i}}),
	\end{equation*} where $\det(A_{\hat{j}\hat{i}})$ is the determinant of the $(m-1) \times (m-1)$ sub-matrix of $A$ obtained by deletion of row $j$ and column $i$. The adjunct matrix of $A$ is the transpose of the so-called cofactor matrix $\cof(A)$ of $A$.
\end{definition}

Further information about the adjunct may be found in \cite{HJ85}. By a straightforward application of the Laplace expansion formula for determinants \cite{HJ85}, one can prove the following proposition.
\begin{proposition} \label{prop:BasicAdj} Let $A \in \R^{m \times m}$, then, with the above definition, we have \begin{enumerate}[{\qquad(}i{)}]
		\item $\adj(A)A = A \adj(A) = \det(A)I_m$,
		\item $\adj(-A) = (-1)^{m-1}\adj(A)$,
		\item $\adj(A^T) = \adj(A)^T = \cof(A)$.
	\end{enumerate}
\end{proposition}

\begin{remark} We note that the adjunct/cofactor of a matrix is not uniquely determined if the matrix is singular. Since we use only the relation \emph{(i)}, we mean by $\adj(A)$ a matrix associated to $A$ that satisfies \emph{(i)}, specified by \eqref{eq:AdjA}. For further details we refer to \cite{Beslin92,Wallace65}. \end{remark}

We recall that the elementary symmetric polynomials $\sigma_h^{(m)}(\lambda)$, $\lambda = (\lambda_1, \dots, \lambda_m)$, are defined by the formula \begin{equation*}
	\sigma_h^{(m)}(\lambda) = (-1)^h \sum_{1 \leq i_1 < i_2 <\dots < i_h \leq m} \lambda_{i_1} \cdot \dots \cdot \lambda_{i_h}
\end{equation*} for $1 \leq h \leq m$ and $\sigma_0^{(m)}(\lambda) = 1$. Using the definition of $\sigma_h^{(m)}(\lambda)$, we get \begin{equation} \label{eq:CharPolDef}
\prod_{h=1}^m (\tau-\lambda_h) = \sum_{h=0}^m \sigma_h^{(m)}(\lambda) \tau^{m-h} = \det(I_m \tau - A) = \sum_{h=0}^m c_h \tau^{m-h},
\end{equation} where $\lambda_1$, \dots, $\lambda_m$ are the eigenvalues of $A$ and $c_{h}=\sigma_h^{(m)}(\lambda)$ for $0 \leq h \leq m$. It is clear that \begin{equation*}
\sigma_1^{(m)}(\lambda) = c_{1} = -\tr(A), \quad \sigma_{m}^{(m)}(\lambda) = c_m = (-1)^m \det(A).
\end{equation*}

The next lemma plays a key role in Section \ref{sec:BlockDiagonal}.
 \begin{lemma} \label{lem:AdjunctRep} Let $A \in \R^{m \times m}$, then the following formulas hold true \begin{eqnarray}
		\label{eq:CH} && \det(A - A)  = A^m + c_{1} A^{m-1} + \dots + c_{m-1} A + c_m I_m = 0, \\
		\label{eq:AdjA} && \adj(A) = (-1)^{m-1} (A^{m-1} + c_{1}A^{m-2} + c_2 A^{m-3} + \dots + c_{m-1} I_m ), \\
		\label{eq:AdjSym}&& \adj(I_m\tau - A) = \sum_{h = 1}^{m} \left[ \sum_{h' = 0}^{h-1} c_{h'} A^{h-h'-1}  \right] \tau^{m-h}.
	\end{eqnarray}
\end{lemma}
Note that formula \eqref{eq:CH} is just the well known Cayley-Hamilton theorem (see for instance \cite{HJ85}). The other two formulas follow from a variant of its proof.
\begin{proof} 
	We consider the product $\adj(I_m \tau - A)(I_m\tau-A)$. By Proposition \ref{prop:BasicAdj}, we have \begin{eqnarray} \label{eq:auxApAdj1}
		\adj(I_m \tau - A)(I_m\tau-A) &=& \det(I_m\tau-A)I_m
	\end{eqnarray} Since the entries of $\adj(I_m \tau - A)$ are, bey Definition \ref{def:Adjunct}, all polynomials of order $\leq m-1$ in $\tau$, we can collect the coefficients in matrices and write \begin{equation*}
	\adj(I_m \tau - A) = \sum_{h=1}^{m} B_{m-h}\tau^{m-h}.
\end{equation*} Plugging this into the left-hand-side of \eqref{eq:auxApAdj1}, we get \begin{equation*}
\sum_{h=1}^{m} B_{m-h}\tau^{m-h+1} - \sum_{h=1}^{m} B_{m-h}A \tau^{m-h} = \sum_{h=0}^m c_{h}I_m \tau^{m-h},
\end{equation*} where we use \eqref{eq:CharPolDef}. Thus, \begin{equation} \label{eq:ApAux2}
\tau^m B_{m-1} + \sum_{h=1}^{m-1} \tau^{m-h}(B_{m-h-1} - B_{m-h}A)  - B_0 A = \sum_{h=0}^m c_{h}I_m \tau^{m-h}. 
\end{equation} A comparison of the coefficients leads to: \begin{table}[H]
\centering
\begin{tabular}{c|c|cl}
	& Coeff. left-hand side \eqref{eq:ApAux2} & Coeff. right-hand side \eqref{eq:ApAux2} \\
	\hline
	$\tau^m$ & $B_{m-1}$ & $ c_0 I_m$  \\
	$\tau^{m-h}$ & $B_{m-h-1}-B_{m-h}A$ & $ c_{h}I_m$,~$1\leq h \leq m-1$\\
	$\tau^0$ & $-B_0A$ & $ c_m I_m$
\end{tabular}
\end{table} If one multiplies the coefficients of $\tau^{m-h}$ with $A^{m-h}$ for $0\leq h \leq m$ and sums them up for $h$ from $0$ to $m$, the sum over the middle column telescopes and adds up to zero which proves \eqref{eq:CH}. If we multiply the coefficients of $\tau^{m-h}$ by $A^{m-1-h}$ for $0 \leq h \leq m-1$, we get, summing up over $h$ from $0$ to $m-1$ that the middle column telescopes and leaves $B_0$. With the sum over the right column, we obtain \begin{equation*}
B_0  = \sum_{h = 0}^{m-1} c_{h}A^{m-1-h}.
\end{equation*}

By the comparison of coefficients, we obtained $-B_0 A = -AB_0 = c_m I_m = (-1)^m \det(A) I_m$, where the second equal sign can be proven by reversing the order of multiplication in \eqref{eq:auxApAdj1}. Thus, we have \begin{equation*}
\adj(A) = (-1)^{m-1} B_0 = (-1)^{m-1} \sum_{h = 0}^{m-1} c_{h}A^{m-1-h}.
\end{equation*} Hence, \eqref{eq:AdjA} is proven. Now we can obtain the $B_{i}$, $i=1,\dots,m$ by multiplying the coefficients of $\tau^{m-h}$ by $A^{m-(i+1)-h}$ for $0 \leq h \leq m-(i+1)$ and summing the equated middle and right column from $0$ to $m-(i+1)$, we obtain \begin{equation*}
B_i = \sum_{h=0}^{m-(i+1)} c_{h} A^{m-(i+1)-h}, 
\end{equation*} and, thus, \begin{equation*}
\adj(I_m \tau - A) = \sum_{h = 1}^{m} \left[ \sum_{h' = 0}^{h-1} c_{h'} A^{h-h'-1}  \right] \tau^{m-h}
\end{equation*} Hence we get \eqref{eq:AdjSym} and the lemma is proven. $\blacksquare$
\end{proof} 

\begin{example} We consider $m=2$. From \eqref{eq:AdjSym}, we have \begin{eqnarray*}
		\adj(I_2 \tau - A(t,\xi)) &=& \sum_{h=1}^{2} \left[ \sum_{h'=0}^{h-1} c_{h'} A^{h-h'-1} \right] \tau^{2-h} \\
		&=& c_0 \tau + (c_0 A + c_1 I_2) \tau^0 = I_2 \tau - \adj(A), 
	\end{eqnarray*} where we used the representation $\adj(A) = -(A - \tr(A) I_2)$ from formula \eqref{eq:AdjA} and $c_1 = \sigma_1^{(2)}(\lambda) = -\tr(A)$, $c_0 = \sigma_0^{(2)}(\lambda) = 1$. This also coincides with our computations in Section \ref{sec:ExCase2}. \end{example}

\begin{example} \label{ex:Adjm3} We consider $m=3$. Now we get \begin{eqnarray*}
		\adj(I_3 \tau - A(t,\xi)) &=& \sum_{h=1}^{3} \left[ \sum_{h'=0}^{h-1} c_{h'} A^{h-h'-1} \right] \tau^{3-h} \\
		&=& c_0 I_3 \tau^2 + (c_0 A + c_1 I_3) \tau + (c_0 A^2 + c_1 A + c_2 I_3) \tau^0 \\
		&=& I_3 \tau_2 + (A-\tr(A)I_3) \tau + \adj(A),
	\end{eqnarray*} where we used $\adj(A) = A^2 + c_1A + c_2I_3$ from \eqref{eq:AdjA} and the coefficients of the characteristic polynomial of $A$
	\begin{eqnarray*}
	 c_0 = 1, \quad c_1 = -\tr(A), \quad c_2 = a_{11}a_{22}+a_{11}a_{33} + a_{22}a_{33} - a_{12}a_{21} - a_{13}a_{31} - a_{23}a_{32}.
\end{eqnarray*} 
This result coincides with our computations in Section \ref{sec:ExCase3}. \end{example}

\end{document}